\documentclass[a4paper]{amsart}
\usepackage{wrapfig}
\usepackage[dvips]{graphicx}
\usepackage{amsmath,amsthm,amssymb,amscd}
\usepackage{mathrsfs}
\usepackage[all]{xy}
\theoremstyle{definition}
\newtheorem{thm}{Theorem}[section]
\newtheorem{Def}[thm]{Definition}
\newtheorem{pro}[thm]{Proposition}
\newtheorem{cor}[thm]{Corollary}
\newtheorem{lem}[thm]{Lemma}

\newtheorem{rem}[thm]{Remark}
\newtheorem{que}[thm]{Question}
\newtheorem*{mainthm}{Theorem A}
\newtheorem*{mainthm2}{Theorem B}
\theoremstyle{definition}

\begin{document}
\title{A characterization of the Razak-Jacelon algebra}
\author{Norio Nawata}
\address{Department of Pure and Applied Mathematics, Graduate School of Information Science and Technology, Osaka University, Yamadaoka 1-5, Suita, Osaka 565-0871, Japan}
\email{nawata@ist.osaka-u.ac.jp}
\thanks{This work was supported by JSPS KAKENHI Grant Number 20K03630}

\begin{abstract}
Combining Elliott, Gong, Lin and Niu's result and  Castillejos and Evington's result, we see that 
if $A$ is a simple separable nuclear monotracial C$^*$-algebra, then  $A\otimes\mathcal{W}$ 
is isomorphic to $\mathcal{W}$ where $\mathcal{W}$ is the Razak-Jacelon algebra. 
In this paper, we give another proof of this. In particular, we show that 
if $\mathcal{D}$ is a simple separable nuclear monotracial $M_{2^{\infty}}$-stable C$^*$-algebra 
which is $KK$-equivalent to $\{0\}$, then $\mathcal{D}$ is isomorphic to $\mathcal{W}$ without 
considering tracial approximations of C$^*$-algebras with finite nuclear dimension. 
Our proof is based on Matui and Sato's technique, Schafhauser's idea in his proof of the 
Tikuisis-White-Winter theorem and properties of Kirchberg's central sequence C$^*$-algebra 
$F(\mathcal{D})$ of $\mathcal{D}$. 
Note that some results for $F(\mathcal{D})$ are based on Elliott-Gong-Lin-Niu's stable 
uniqueness theorem. 
Also, we characterize $\mathcal{W}$ by using properties of $F(\mathcal{W})$. 
Indeed, we show that a simple separable nuclear monotracial 
C$^*$-algebra $D$ is isomorphic to $\mathcal{W}$ if and only if 
$D$ satisfies the following properties: \ \\
(i) for any $\theta\in [0,1]$, there exists a projection $p$ in $F(D)$ such that 
$\tau_{D, \omega}(p)=\theta$, \ \\
(ii) if $p$ and $q$ are projections in $F(D)$ such that $0<\tau_{D, \omega}(p)=\tau_{D, \omega}(q)$, 
then  $p$ is Murray-von Neumann equivalent to $q$,  \ \\
(iii) there exists an injective homomorphism from $D$ to $\mathcal{W}$. 
\end{abstract}
\maketitle

\section{Introduction} 

The Razak-Jacelon algebra $\mathcal{W}$ is a certain simple separable nuclear 
monotracial C$^*$-algebra which is $KK$-equivalent to $\{0\}$. Note that 
such a C$^*$-algebra must be stably projectionless, that is, 
$\mathcal{W}\otimes M_{n}(\mathbb{C})$ has no non-zero projections for any $n\in\mathbb{N}$. 
In particular, every stably projectionless C$^*$-algebra is non-unital. 
In \cite{J}, Jacelon constructed $\mathcal{W}$ as an inductive limit C$^*$-algebra 
of Razak's building blocks \cite{Raz}. 
We can regard $\mathcal{W}$ as a stably finite analogue of the Cuntz algebra $\mathcal{O}_2$. 
In particular, $\mathcal{W}$ is expected to play a central role in the classification theory of 
simple separable nuclear stably projectionless C$^*$-algebras as $\mathcal{O}_2$ 
played in the classification theory of Kirchberg algebras (see, for example, \cite{Ror1} and \cite{G2}). 
We refer the reader to \cite{EGLN0}, \cite{EGLN} and \cite{GL2} for recent progress in 
the classification of simple separable nuclear stably projectionless C$^*$-algebras. 
Note that there exist many interesting examples of simple stably projectionless C$^*$-algebras. 
See, for example, \cite{Connes2}, \cite{Ell}, \cite{K2}, \cite{KK1}, \cite{KK2} and \cite{Rob}. 

Combining Elliott, Gong, Lin and Niu's result \cite{EGLN} and  Castillejos and Evington's result 
\cite{CE} (see also \cite{CETWW}), we see that if $A$ is a simple separable nuclear 
monotracial C$^*$-algebra, then  $A\otimes\mathcal{W}$ is isomorphic to $\mathcal{W}$. 
This can be considered as a Kirchberg-Phillips type absorption theorem \cite{KP} for $\mathcal{W}$. 
In this paper, we give another proof of this. 
In our proof, we do not consider tracial approximations of C$^*$-algebras with 
finite nuclear dimension. Also, we mainly consider abstract settings and do not use 
any classification theorem based on inductive limit structures of $\mathcal{W}$ 
other than Razak's classification theorem \cite{Raz}. 
(Actually, we need Razak's classification theorem only for 
$\mathcal{W}\otimes M_{2^{\infty}}\cong \mathcal{W}$.) 
We obtain a Kirchberg-Phillips type absorption theorem for $\mathcal{W}$ as a corollary of 
the following theorem. 

\begin{mainthm} (Theorem \ref{thm:main}) \ \\
Let $\mathcal{D}$ be a simple separable nuclear monotracial $M_{2^{\infty}}$-stable C$^*$-algebra 
which is $KK$-equivalent to $\{0\}$. Then $\mathcal{D}$ is isomorphic to $\mathcal{W}$. 
\end{mainthm}

Our proof of the theorem above is based on Matui and Sato's technique \cite{MS}, 
\cite{MS2}, \cite{MS3}, Schafhauser's  idea \cite{Sc} (see also \cite{Sc2}) 
in his proof of the Tikuisis-White-Winter theorem \cite{TWW} and properties of 
Kirchberg's central sequence C$^*$-algebra $F(\mathcal{D})$ of $\mathcal{D}$. 

Matui-Sato's technique enables us to show that certain (relative) central sequence 
C$^*$-algebras have strict comparison. 
Note that a key concept in their technique is property (SI). This concept was 
introduced by Sato in \cite{Sa0} and \cite{Sa}. 

Borrowing Schafhauser's  idea, we show that if $\mathcal{D}$ is a simple separable 
nuclear monotracial ($M_{2^{\infty}}$-stable) C$^*$-algebra which is $KK$-equivalent to $\{0\}$, 
then there exist ``trace preserving'' homomorphisms from $\mathcal{D}$ to ultrapowers 
$B^{\omega}$ of certain C$^*$-algebras $B$. 
Combining this and a uniqueness result for approximate homomorphisms from $\mathcal{D}$, 
we obtain an existence result, that is, existence of homomorphisms from 
$\mathcal{D}$ to certain C$^*$-algebras. 
Schafhauser's arguments are based on the extension theory (or $KK$-theory) and 
Elliott and Kucerovsky's result \cite{EllK} with a correction by Gabe \cite{G}. 
Hence Schafhauser's  arguments are suitable for our purpose, that is, a study of  
C$^*$-algebras which are $KK$-equivalent to $\{0\}$. 

We studied properties of $F(\mathcal{W})$ in \cite{Na3} and \cite{Na} by using 
Elliott-Gong-Lin-Niu's stable uniqueness theorem in \cite{EGLN}. 
In particular, we showed that $F(\mathcal{W})$ has many projections and 
satisfies a certain comparison theory for projections. 
By these properties and Connes' $2\times 2$ matrix trick, we can show that 
every trace preserving endomorphism of $\mathcal{W}$ is approximately inner. 
(Note that Jacelon showed this result in \cite[Corollary 4.6]{J} as an application of Razak's results 
in \cite{Raz}.)
This argument is a traditional argument in the theory of operator algebras (see \cite{Connes}). 
In this paper, we remark that arguments in \cite{Na3} and \cite{Na} work for 
a simple separable nuclear monotracial $M_{2^{\infty}}$-stable C$^*$-algebra $\mathcal{D}$ 
which is $KK$-equivalent to $\{0\}$. Also, we characterize $\mathcal{W}$ by 
using these properties of $F(\mathcal{W})$. 
Indeed, we show the following theorem. 

\begin{mainthm2} (Theorem \ref{thm:main2}) \ \\
Let $D$ be a simple separable nuclear monotracial C$^*$-algebra. Then 
$D$ is isomorphic to $\mathcal{W}$ if and only if $D$ satisfies the following properties: 
\ \\
(i) for any $\theta\in [0,1]$, there exists a projection $p$ in $F(D)$ such that 
$\tau_{D, \omega}(p)=\theta$, \ \\
(ii) if $p$ and $q$ are projections in $F(D)$ such that $0<\tau_{D, \omega}(p)=\tau_{D, \omega}(q)$, 
then  $p$ is Murray-von Neumann equivalent to $q$,  \ \\
(iii) there exists an injective homomorphism from $D$ to $\mathcal{W}$. 
\end{mainthm2}

This paper is organized as follows. In Section \ref{sec:pre}, we collect notations, definitions and 
some results. In particular, we recall Matui-Sato's technique. 
In Section \ref{sec:property-w}, we introduce the property W, which is a key property for 
uniqueness results. Also, we remark that arguments in \cite{Na3} and \cite{Na} 
work for more general settings. 
In Section \ref{sec:uniqueness}, we show uniqueness results. First, 
we show that if $D$ has property W, then every trace preserving endomorphism of $D$ 
is approximately inner. Secondly, we consider 
a uniqueness theorem for approximate homomorphisms from 
a simple separable nuclear monotracial $M_{2^{\infty}}$-stable C$^*$-algebra $\mathcal{D}$ 
which is $KK$-equivalent to $\{0\}$ for an existence result in Section \ref{sec:existence}. 
In Section \ref{sec:existence}, we show an existence result by borrowing Schafhauser's  idea.
In Section \ref{sec:main}, we show the main results in this paper.

\section{Preliminaries}\label{sec:pre}

In this section we shall collect notations, definitions and some results. 
We refer the reader to \cite{Bla} and \cite{Ped2} for basics of operator algebras. 

For a C$^*$-algebra $A$, we denote by $A_{+}$ the sets of positive elements of $A$ and 
by $A^{\sim}$ the unitization algebra of $A$. Note that if $A$ is unital, then 
$A=A^{\sim}$. For $a,b\in A_{+}$, we say that $a$ is \textit{Murray-von Neumann equivalent 
to} $b$, written $a\sim b$, if there exists an element $z$ in $A$ such that 
$z^*z=a$ and $zz^*=b$. Note that $\sim$ is an equivalence relation by \cite[Theorem 3.5]{Ped3}. 
For $a,b\in A$, we denote by $[a,b]$ the commutator $ab-ba$. 
For a subset $F$ of $A$ and $\varepsilon>0$, we say that a completely positive (c.p.) map 
$\varphi :A\to B$ is \textit{$(F, \varepsilon)$-multiplicative} if 
$$
\|\varphi (ab)-\varphi(a)\varphi(b)\|<\varepsilon
$$ 
for any $a, b\in F$. 
Let $\mathcal{Z}$ and $M_{2^\infty}$ 
denote the Jiang-Su algebra and the CAR algebra, respectively. 
We say a C$^*$-algebra $A$ is \textit{monotracial} if $A$ has a unique tracial state 
and no unbounded traces. 
In the case where $A$ is monotracial, we denote by $\tau_{A}$ the unique tracial state on 
$A$ unless otherwise specified. 

\subsection{Razak-Jacelon algebra $\mathcal{W}$}

The \textit{Razak-Jacelon algebra} $\mathcal{W}$ is a certain simple separable nuclear 
monotracial C$^*$-algebra which is $KK$-equivalent to $\{0\}$.  
In \cite{J}, $\mathcal{W}$ is constructed as an inductive limit C$^*$-algebra 
of Razak's building blocks. 
By Razak's classification theorem \cite{Raz}, 
$\mathcal{W}$ is $M_{2^{\infty}}$-stable, and hence $\mathcal{W}$ is $\mathcal{Z}$-stable. 
In this paper, we do not assume any classification theorem for $\mathcal{W}$ other than 
Razak's classification theorem.

\subsection{Kirchberg's central sequence C$^*$-algebras}\label{sec:pre-kir}

We shall recall the definition of Kirchberg's central sequence C$^*$-algebras in \cite{Kir2}. 
Fix a free ultrafilter $\omega$ on $\mathbb{N}$. For a C$^*$-algebra $B$, put 
$$
c_{\omega}(B):=\{\{x_n\}_{n\in\mathbb{N}}\in \ell^{\infty}(\mathbb{N}, B)\; 
|\; \lim_{n \to \omega}\| x_n\| =0 \}, \; 
B^{\omega}:=\ell^{\infty}(\mathbb{N}, B)/c_{\omega}(B). 
$$
We denote by $(x_n)_n$ a representative of an element in $B^{\omega}$. 
Let $A$ be a C$^*$-subalgebra of $B^{\omega}$. 
Set 
$$
\mathrm{Ann}(A,B^{\omega}):=\{(x_n)_n\in B^{\omega}\cap A^{\prime}\; |\; (x_n)_na =0
\;\mathrm{for}\;\mathrm{any}\; a\in A \}.
$$
Then $\mathrm{Ann}(A,B^{\omega})$ is a closed ideal of $B^{\omega}\cap A^{\prime}$. 
Define a \textit{(relative) central sequence C$^*$-algebra} $F(A, B)$ 
of $A\subseteq B^{\omega}$ by 
$$
F(A, B):=B^{\omega}\cap A^{\prime}/\mathrm{Ann}(A, B^{\omega}).
$$ 
We identify $B$ with the C$^*$-subalgebra of $B^\omega$ consisting of equivalence 
classes of constant sequences.
In the case $A=B$, we denote $F(B, B)$ by $F(B)$ and call it the 
\textit{central sequence C$^*$-algebra of $B$}. 
If $A$ is $\sigma$-unital, then $F(A, B)$ is unital by \cite[Proposition 1.9]{Kir2}.  
Indeed, let $s=(s_n)_n$ be a strictly positive element in $A\subseteq B^{\omega}$. 
Since we have $\lim_{k\to \infty}s^{1/k}s=s$, taking a suitable sequence 
$\{k(n)\}_{n\in\mathbb{N}}\subset \mathbb{N}$, 
we obtain $s^{\prime}=(s_{n}^{1/k(n)})_n\in B^{\omega}$ such that $s^{\prime}s=s$. 
Then it is easy to see that $s^{\prime}\in B^{\omega}\cap A^{\prime}$ and 
$[s^{\prime}]=1$ in $F(A, B)$. 
Note that the inclusion $B\subset B^{\sim}$ induces an isomorphism from $F(A, B)$ onto 
$F(A, B^{\sim})$ because we have $[xs^{\prime}]=[x]$ in $F(A, B^{\sim})$ for any 
$x\in (B^{\sim})^{\omega}\cap A^{\prime}$. 

Let $\tau_{B}$ be a tracial state on $B$. Define $\tau_{B, \omega}: B^{\omega}\to \mathbb{C}$ 
by $\tau_{B, \omega} ((x_n)_n) = \lim_{n\to\omega} \tau_{B}(x_n)$ for any $(x_n)_n\in B^{\omega}$. 
Since $\omega$ is an ultrafilter, it is easy to see that $\tau_{B, \omega}$ is a well-defined 
tracial state on $B^{\omega}$. 
The following proposition is a relative version of \cite[Proposition 2.1]{Na3}.

\begin{pro}\label{pro:trace-induce}
Let $B$ be a C$^*$-algebra with a faithful tracial state $\tau_{B}$, 
and let $A$ be a C$^*$-subalgebra of $B^{\omega}$. 
Assume that  $\tau_{B, \omega}|_{A}$ is a state. 
Then $\tau_{B, \omega}((x_n)_n)=0$ for any $(x_n)_n\in \mathrm{Ann}(A,B^{\omega})$. 
\end{pro}
\begin{proof}
Let $\{h_{\lambda}\}_{\lambda\in \Lambda}$ be an approximate unit for $A$. 
Since $\tau_{B, \omega}|_{A}$ is a state, 
we have $\lim\tau_{B, \omega} (h_{\lambda})=1$. 
The rest of proof is same as the proof of \cite[Proposition 2.1]{Na3}. 
\end{proof}
By the proposition above, 
if $\tau_{B, \omega}|_{A}$ is a state, then $\tau_{B, \omega}$ induces a tracial state on $F(A, B)$. 
We denote it by the same symbol $\tau_{B, \omega}$ for simplicity.

\subsection{Invertible elements in unitization algebras}

Let $\mathrm{GL}(A^{\sim})$ denote the set of invertible elements in $A^{\sim}$. 
The following proposition is trivial if $1_{A^{\sim}}=1_{B^{\sim}}$. 
\begin{pro}\label{pro:inclusion-gl}
Let $A\subseteq B$ be an inclusion of C$^*$-algebras. Then 
$\mathrm{GL}(A^{\sim})\subset \overline{\mathrm{GL}(B^{\sim})}$. 
\end{pro}
\begin{proof}
Let $x\in \mathrm{GL}(A^{\sim})$. 
There exists $\varepsilon_0>0$ such that for any $0\leq \varepsilon < \varepsilon_0$, 
$x+\varepsilon 1_{A^{\sim}}\in \mathrm{GL}(A^{\sim})$ because $\mathrm{GL}(A^{\sim})$ 
is open. Since we have $\mathrm{Sp}_A(x)\cup \{0\}=\mathrm{Sp}_B(x)\cup\{0\}$, 
$x+\varepsilon 1_{B^{\sim}}\in \mathrm{GL}(B^{\sim})$ for any $0< \varepsilon < \varepsilon_0$. 
Therefore $x\in \overline{\mathrm{GL}(B^{\sim})}$. 
\end{proof}

The following corollary is an immediate consequence of the proposition above. 

\begin{cor}\label{cor:inclusion-gl}
Let $\{A_n\}_{n\in\mathbb{N}}$ be a sequence of C$^*$-algebras with $A_{n}\subseteq A_{n+1}$, 
and let $A=\overline{\bigcup_{n=1}^\infty A_n}$. 
If $A_n\subseteq \overline{\mathrm{GL}(A_n^{\sim})}$ for any $n\in\mathbb{N}$, 
then $A \subseteq \overline{\mathrm{GL}(A^{\sim})}$. 
\end{cor}

The following proposition is well-known if $B$ is unital. See, for example, 
the proof of \cite[Proposition 3.2]{Sc}.

\begin{pro}\label{pro:non-sep-gl}
Let $B$ be a C$^*$-algebra with $B\subseteq \overline{\mathrm{GL}(B^{\sim})}$. Then 
$B^{\omega}\subseteq \overline{\mathrm{GL}((B^{\omega})^{\sim})}$. 
\end{pro}
\begin{proof}
We shall show only the case where $B$ is non-unital. 
Let $(x_n)_n\in B^{\omega}$. Because of $B\subseteq \overline{\mathrm{GL}(B^{\sim})}$, 
there exists $(z_n)_n\in (B^{\sim})^{\omega}$ such that $z_n\in \mathrm{GL}(B^{\sim})$ 
for any $n\in\mathbb{N}$ and $(x_n)_n=(z_n)_n$ in $(B^{\sim})^{\omega}$. 
For any $n\in\mathbb{N}$, put $u_n:=z_n(z_n^*z_n)^{-1/2}$. Then $u_n$ is a unitary element and 
$z_n=u_n(z_n^*z_n)^{1/2}$. Note that we have $(x_n)_n= (u_n)_n(x_n^*x_n)_n^{1/2}$. 
For any $n\in\mathbb{N}$, there exist $y_n\in B$ and $\lambda_n\in\mathbb{C}$ such that 
$u_n=y_n+\lambda_n 1_{B^{\sim}}$ and $|\lambda_n|=1$ because $u_n$ is a unitary element 
in $B^{\sim}$. Since $\omega$ is an ultrafilter, there exists $\lambda_0\in\mathbb{C}$ such that 
$\lim_{n\to\omega}\lambda_n=\lambda_0$. 
Hence 
$$
(u_n)_n= (y_n)_n+ \lambda_0 1_{(B^{\omega})^{\sim}}\in (B^{\omega})^{\sim}.
$$ 
Since we have 
$$
((y_n)_n+ \lambda_0 1_{(B^{\omega})^{\sim}})((x_n^*x_n)_n^{1/2}+\varepsilon 1_{(B^{\omega})^{\sim}})
\to (x_n)_n
$$
as $\varepsilon\to 0$, $(x_n)_n\in \overline{\mathrm{GL}((B^{\omega})^{\sim})}$. 
\end{proof}

Note that if $B$ has almost stable rank one (see \cite{Rob2} for the definition), then 
$B\subseteq \overline{\mathrm{GL}(B^{\sim})}$. Also, if $B$ is unital, then 
$B\otimes\mathbb{K}\subseteq \overline{\mathrm{GL}((B\otimes\mathbb{K})^{\sim}})$ where 
$\mathbb{K}$ is the C$^*$-algebra of compact operators on an infinite-dimensional separable 
Hilbert space. 

\subsection{Matui-Sato's technique}\label{sec:MS}

We shall review Matui and Sato's technique in \cite{MS}, \cite{MS2} and \cite{MS3}. 
Let $B$ be a monotracial C$^*$-algebra, 
and let $A$ be a simple separable nuclear monotracial C$^*$-subalgebra of $B^{\omega}$. 
Assume that $\tau_{B}$ is faithful and $\tau_{B, \omega}|_{A}$ is a state. 
Consider the Gelfand-Naimark-Segal (GNS) representation $\pi_{\tau_B}$ of 
$B$ associated with $\tau_{B}$, and put 
$$
M:= \ell^{\infty}(\mathbb{N}, \pi_{\tau_B}(B)^{''})/\{\{x_n\}_{n\in\mathbb{N}}\; |\; 
\tilde{\tau}_{B, \omega}((x_n^*x_n)_n):=\lim_{n\to\omega} \tilde{\tau}_B(x_n^*x_n)=0 \}
$$
where $ \tilde{\tau}_B$ is the unique normal extension of $\tau_B$ on $\pi_{\tau_B}(B)^{''}$. 
Note that $M$ is a von Neumann algebraic ultrapower of $\pi_{\tau_B}(B)^{''}$ and  
$\tilde{\tau}_{B,\omega}$ is a faithful normal tracial state on $M$. 
Since $B$ is monotracial, $\pi_{\tau_B}(B)^{''}$ is a finite factor, and hence 
$M$ is also a finite factor. 
Define a homomorphism $\varrho$ from $B^{\omega}$ to $M$ by 
$\varrho ((x_n)_n)=(\pi_{\tau_B}(x_n))_n$. Kaplansky's density theorem implies 
that $\varrho$ is surjective. 
Moreover, \cite[Theorem 3.1]{MS3} (see also \cite[Theorem 3.3]{KR}) implies that 
the restriction $\varrho$ on $B^{\omega}\cap A^{\prime}$ is a surjective homomorphism 
onto $M\cap \varrho (A)^{\prime}$.

\begin{pro}\label{pro:factor}
With notation as above, $M\cap \varrho (A)^{\prime}$ is a finite factor. 
\end{pro}
\begin{proof} 
Note that $\tilde{\tau}_{B,\omega}$ is the unique tracial state on $M$ since $M$ is a finite factor. 
It is enough to show that $M\cap \varrho (A)^{\prime}$ is monotracial. Let $\tau$ be a tracial state 
on $M\cap \varrho (A)^{\prime}$. 
Since we assume that $\tau_{B, \omega}|_{A}$ is a state, we see that if $A$ is unital, 
then $\varrho (1_A)= 1_{M}$. Hence $\varrho$ can be extended to a unital homomorphism 
$\varrho^{\sim}$ from $A^{\sim}$ to $M$, and 
$M\cap \varrho (A)^{\prime}= M\cap \varrho^{\sim}(A^{\sim})^{\prime}$.
By \cite[Lemma 3.21]{BBSTWW}, there exists a positive element $a$ in $A^{\sim}$ such that 
$\tilde{\tau}_{B, \omega}(\varrho^{\sim} (a))=1$ and 
$\tau (x)=\tilde{\tau}_{B, \omega}(\varrho^{\sim}(a)x)$ for any 
$x\in M\cap \varrho (A)^{\prime}$. Since $A$ is monotracial, 
$$
\tau (x)=\tilde{\tau}_{B, \omega}(\varrho^{\sim}(a)x)=\tilde{\tau}_{B, \omega}(\varrho^{\sim}(a))
\tilde{\tau}_{B, \omega}(x)=\tilde{\tau}_{B,\omega}(x).
$$
Indeed, let $x_0$ be a positive contraction 
in $M\cap \varrho (A)^{\prime}$. For any $a\in A$, define 
$\tau^{\prime}(a):=\tilde{\tau}_{B, \omega}(\varrho(a)x_0)$. 
Then $\tau^{\prime}$ is a tracial positive linear functional on $A$. 
Since $A$ is monotracial and $\tau_{B,\omega}|_A$ is a tracial state on $A$, 
there exists a positive number $t$ such that $\tau^{\prime}(a)= t \tau_{B, \omega}(a)$ for any 
$a\in A$. 
Note that if $\{h_n\}_{n\in\mathbb{N}}$ is an approximate unit for $A$, 
then $t=\lim_{n\to \infty}\tau^{\prime}(h_n)$. On the other hand, we have 
\begin{align*}
|\tilde{\tau}_{B, \omega}(x_0)-\tau^{\prime}(h_n)|
& =|\tilde{\tau}_{B, \omega}((1-\varrho(h_n))x_0)|=
|\tilde{\tau}_{B, \omega}((1-\varrho(h_n))^{1/2}x_0(1-\varrho(h_n))^{1/2})| \\
& \leq 
|\tilde{\tau}_{B, \omega}(1-\varrho (h_n))|=|1-\tau_{B, \omega}(h_n)| \to 0
\end{align*}
as $n\to \infty$. Hence $t=\tilde{\tau}_{B, \omega}(x_0)$, and  
$\tilde{\tau}_{B, \omega}(\varrho(a)x_0)=\tilde{\tau}_{B, \omega}(\varrho(a))
\tilde{\tau}_{B, \omega}(x_0)$ for any $a\in A$. It is easy to see that this implies 
$\tilde{\tau}_{B, \omega}(\varrho^{\sim}(a)x)=\tilde{\tau}_{B, \omega}(\varrho^{\sim}(a))
\tilde{\tau}_{B, \omega}(x)$
for any $a\in A^{\sim}$ and $x\in M\cap \varrho (A)^{\prime}$. Therefore we have 
$\tau (x)= \tilde{\tau}_{B, \omega}(x)$ for any $x\in M\cap \varrho (A)^{\prime}$. 
Consequently, $M\cap \varrho (A)^{\prime}$ is monotracial. 
\end{proof}

For $a,b\in A_{+}$, we say that $a$ is \textit{Cuntz smaller than} $b$, written $a\precsim b$, 
if there exists a sequence $\{x_n\}_{n\in\mathbb{N}}$ of $A$ such that 
$\| x_n^*bx_n-a\|\rightarrow 0$. 
A monotracial C$^*$-algebra $B$ is said to have 
\textit{strict comparison} if for any $k\in\mathbb{N}$,  
$a,b\in M_k(B)_{+}$ with $d_{\tau_B\otimes\mathrm{Tr}_k}(a)< d_{\tau_B\otimes\mathrm{Tr_k}} (b)$ 
implies $a\precsim b$ 
where $\mathrm{Tr}_k$ is the unnormalized trace on $M_{k}(\mathbb{C})$ and 
$d_{\tau_B\otimes\mathrm{Tr}_k}(a)=\lim_{n\to\infty}\tau_B\otimes\mathrm{Tr}_k(a^{1/n})$. 
Using \cite[Lemma 5.7]{Na2}, essentially the same proofs as \cite[Lemma 3.2]{MS3} and 
\cite[Theorem 1.1]{MS} show the following proposition. 
See also the proof of \cite[Lemma 3.6]{Na}. 
\begin{pro}\label{pro:si}
Let $B$ be a monotracial C$^*$-algebra, and let $A$ be a simple 
separable non-type I nuclear monotracial C$^*$-subalgebra of $B^{\omega}$. 
Assume that $\tau_{B}$ is faithful, $\tau_{B, \omega}|_A$ is a state and $B$ has strict comparison. 
Then 
$B$ has property (SI) relative to $A$, that is, for any positive contractions 
$a$ and $b$ in $B^{\omega}\cap A^{\prime}$ satisfying  
$$
\tau_{B, \omega} (a)=0 \quad \text{and} \quad \inf_{m\in\mathbb{N}} \tau_{B, \omega} (b^m)>0,
$$
there exists an element $s$ in $B^{\omega}\cap A^{\prime}$ such that  $s^*s=a$ and $bs=s$. 
\end{pro}

By Proposition \ref{pro:trace-induce}, $\varrho$ induces a surjective homomorphism from  
$F(A, B)$ to $M\cap \varrho(A)^{\prime}$. 
We denote it by the same symbol $\varrho$ for simplicity. 
Using Proposition \ref{pro:factor} and Proposition \ref{pro:si}, essentially the same proofs as 
\cite[Proposition 3.3]{MS3} and \cite[Proposition 4.8]{MS2} show the following proposition. 
See also the proof of \cite[Proposition 3.8]{Na}. 

\begin{pro}\label{pro:central-strict-comparison}
Let $B$ be a monotracial C$^*$-algebra, and let $A$ be a simple 
separable non-type I nuclear monotracial C$^*$-subalgebra of $B^{\omega}$. 
Assume that $\tau_{B}$ is faithful, $\tau_{B, \omega}|_A$ is a state and $B$ has strict comparison. Then 
$F(A, B)$ is monotracial and has strict comparison. Furthermore, if $a$ and $b$ are positive elements 
in $F(A,B)$ satisfying $d_{\tau_{B,\omega}}(a)< d_{\tau_{B, \omega}}(b)$, then there exists an element 
$r$ in $F(A, B)$ such that $r^*br=a$. 
\end{pro}

\section{Property W}\label{sec:property-w}

In this section we shall introduce the property W, which is a key property 
in Section \ref{sec:uniqueness}. 

\begin{Def}
Let $D$ be  a simple separable nuclear monotracial C$^*$-algebra. We say that 
$D$ has \textit{property W} if $F(D)$ satisfies the following properties: 
\ \\
(i) for any $\theta\in [0,1]$, there exists a projection $p$ in $F(D)$ such that 
$\tau_{D, \omega}(p)=\theta$, \ \\
(ii) if $p$ and $q$ are projections in $F(D)$ such that $0<\tau_{D, \omega}(p)=\tau_{D, \omega}(q)$, 
then  $p$ is Murray-von Neumann equivalent to $q$.
\end{Def}

By arguments in \cite{Na3} and \cite{Na}, we see that if $\mathcal{D}$ is a simple separable nuclear 
monotracial $M_{2^{\infty}}$-stable C$^*$-algebra which is $KK$-equivalent to $\{0\}$, then 
$\mathcal{D}$ has property W. 
We shall give a sketch of a proof for reader's convenience and show a slight generalization 
(or a relative version). 

In this section, we assume that $\mathcal{D}$ is a simple separable nuclear monotracial 
$M_{2^{\infty}}$-stable C$^*$-algebra which is $KK$-equivalent to $\{0\}$ and  
$B$ is a simple monotracial C$^*$-algebra with strict comparison and 
$B\subseteq \overline{\mathrm{GL}(B^{\sim})}$. Let $\Phi$ be a homomorphism from 
$\mathcal{D}$ to $B^{\omega}$ such that $\tau_{\mathcal{D}}=\tau_{B, \omega}\circ \Phi$. 
By the Choi-Effros lifting theorem, there exists a 
sequence $\{\Phi_n\}_{n\in\mathbb{N}}$ of contractive c.p. maps from $\mathcal{D}$ to $B$ such that 
$\Phi (x) =(\Phi_n(x))_n$ for any $x\in\mathcal{D}$. 
Since we assume $\tau_{\mathcal{D}}=\tau_{B, \omega}\circ \Phi$, 
$\tau_{B, \omega}|_{\Phi(\mathcal{D})}$ is a state. 
Hence $\tau_{B, \omega}$ is the unique tracial state on $F(\Phi (\mathcal{D}), B)$ by 
Proposition \ref{pro:central-strict-comparison}. 
The following proposition is an analogous proposition of  
\cite[Proposition 4.2]{Na3} and \cite[Proposition 2.6]{Na}. 

\begin{pro}\label{pro:property-w-ii}
(i) For any $N\in\mathbb{N}$, there exists a unital homomorphism from 
$M_{2^{N}}(\mathbb{C})$ to $F(\Phi (\mathcal{D}), B)$. \ \\
(ii) For any $\theta\in [0,1]$, there exists a projection $p$ in $F(\Phi (\mathcal{D}), B)$ 
such that $\tau_{B, \omega}(p)=\theta$. \ \\
(iii) Let $h$ be a positive element in $F(\Phi (\mathcal{D}), B)$ such that 
$d_{\tau_{B, \omega}}(h)>0$. For any $\theta \in [0, d_{\tau_{B, \omega}}(h))$, 
there exists a non-zero projection $p$ in 
$\overline{hF(\Phi (\mathcal{D}), B)h}$ such that $\tau_{B, \omega}(p)=\theta$. 
\end{pro}
\begin{proof}

(i) Since $\mathcal{D}$ is isomorphic to $\mathcal{D}\otimes M_{2^{\infty}}= 
\mathcal{D}\otimes \bigotimes_{n\in\mathbb{N}} M_{2^N}(\mathbb{C})$,  a similar argument as in the 
proof of \cite[Proposition 4.2]{Na3} shows that there exists a family $\{(e_{ij, m})_m\}_{i,j=1}^{2^N}$ 
of contractions in $\mathcal{D}^{\omega}\cap\mathcal{D}^{\prime}$ such that 
$(\sum_{\ell=1}^{2^{N}}e_{\ell\ell, m}x)_m=x$ and $(e_{ij, m}e_{kl, m}x)_m= (\delta_{jk}e_{il, m}x)_m$ 
for any $1\leq i,j,k,l\leq 2^{N}$ and $x\in \mathcal{D}$. 
Note that we have 
$$
\lim_{m\to\omega} \| ([\Phi_n (e_{ij, m}), \Phi_n(x)])_n \| =0, \quad 
\lim_{m\to \omega} \| (\sum_{\ell=1}^{2^N}\Phi_n (e_{\ell\ell, m})\Phi_n(x)-\Phi_n(x))_n\|=0
$$
and
$$ 
\lim_{m\to\omega} \| ((\Phi_n (e_{ij, m})\Phi_n(e_{kl, m})-
\delta_{jk}\Phi_n(e_{il, m}))\Phi_n(x))_n\|=0
$$
for any $1\leq i,j,k,l\leq 2^{N}$ and  $x\in\mathcal{D}$. 
Hence, for any finite subset $F\subset \mathcal{D}$ and $\varepsilon>0$, there exists 
a family of $\{(\Phi_n(e_{ij, (F, \varepsilon)}))_n\}_{i,j=1}^{2^N}$ of contractions in 
$B^{\omega}$ such that 
$$
\lim_{n\to\omega} \| [\Phi_n (e_{ij, (F, \varepsilon)}), \Phi_n(x)] \| < \varepsilon , \quad
\lim_{n\to \omega} \| \sum_{\ell=1}^{2^N}\Phi_n (e_{\ell\ell, (F,\varepsilon)})
\Phi_n(x)-\Phi_n(x)\|< \varepsilon
$$
and
$$ 
\lim_{n\to\omega} \| (\Phi_n (e_{ij, (F, \varepsilon)})\Phi_n(e_{kl, (F, \varepsilon)})-
\delta_{jk}\Phi_n(e_{il, (F, \varepsilon)}))\Phi_n(x)\| < \varepsilon
$$
for any $1\leq i,j,k,l\leq 2^{N}$ and  $x\in F$. 
Let $\{F_{m}\}_{m\in\mathbb{N}}$ be an increasing sequence of finite subsets in $\mathcal{D}$ 
such that $\mathcal{D}=\overline{\bigcup_{m\in\mathbb{N}} F_{m}}$. 
We can find a sequence $\{X_m\}_{m\in\mathbb{N}}$ of elements in $\omega$ such that 
$X_{m+1}\subset X_{m}$ and for any $n\in X_{m}$, 
$$
\| [\Phi_n (e_{ij, (F_{m}, \frac{1}{m})}), \Phi_n(x)] \| < \frac{1}{m}, \quad 
\| \sum_{\ell=1}^{2^N}\Phi_n (e_{\ell\ell, (F_{m},\frac{1}{m})})
\Phi_n(x)-\Phi_n(x)\|< \frac{1}{m}
$$
and
$$ 
\| (\Phi_n (e_{ij, (F_{m}, \frac{1}{m})})\Phi_n(e_{kl, (F_{m}, \frac{1}{m})})-
\delta_{jk}\Phi_n(e_{il, (F_{m}, \frac{1}{m})}))\Phi_n(x)\| < \frac{1}{m}
$$
for any $1\leq i,j,k,l\leq 2^{N}$ and  $x\in F_{m}$.
For any $1\leq i,j\leq 2^{N}$, put 
$$
E_{ij, n} := \left\{\begin{array}{cl}
0 & \text{if } n\notin X_1   \\
\Phi_n (e_{ij, (F_{m}, \frac{1}{m})}) & \text{if } n\in X_m\setminus X_{m+1}\quad (m\in\mathbb{N})
\end{array}
\right..
$$
Then we have $(E_{ij, n})_n\in B^{\omega}\cap \Phi (\mathcal{D})^{\prime}$, 
$$
\sum_{\ell=1}^{2^N}[(E_{\ell\ell, n})_n]=1 \quad \text{and} \quad 
[(E_{ij, n})_n][(E_{kl, n})_n]=\delta_{jk}[(E_{il, n})_n]
$$
in $F(\Phi (\mathcal{D}), B)$ for any $1\leq i,j,k,l \leq 2^{N}$. 
Therefore there exists a unital homomorphism from $M_{2^{N}}(\mathbb{C})$ to 
$F(\Phi (\mathcal{D}), B)$.

(ii) Since $\mathcal{D}$ is isomorphic to $\mathcal{D}\otimes M_{2^{\infty}}= 
\mathcal{D}\otimes \bigotimes_{n\in\mathbb{N}} M_{2^{\infty}}$, a similar argument 
as in the proof of \cite[Proposition 4.2]{Na3} shows that there exists a positive contraction 
$(p_m)_m$ in $\mathcal{D}^{\omega}\cap \mathcal{D}$ such that $((p_m^2-p_m)x)_m=0$ for any 
$x\in\mathcal{D}$ and $\tau_{\mathcal{D}, \omega}((p_m)_m)=\theta$. 
By a similar argument as above, we obtain a projection $p$ in $F(\Phi (\mathcal{D}), B)$ 
such that $\tau_{B, \omega}(p)=\theta$. 

(iii) Using Proposition \ref{pro:central-strict-comparison} 
instead of \cite[Proposition 4.1]{Na3}, we obtain the conclusion by the same argument as in the 
proof of \cite[Proposition 4.2]{Na3}.
\end{proof}

The proposition above and the same arguments as in \cite[Section 4]{Na3} show the following 
corollary. 

\begin{cor}\label{cor:MvN-u} (cf. \cite[Proposition 4.8]{Na3}).
Let $p$ and $q$ be projections in $F(\Phi (\mathcal{D}), B)$ such that $\tau_{B, \omega} (p)<1$. 
Then $p$ and $q$ are Murray-von Neumann equivalent if and only if $p$ and $q$ are 
unitarily equivalent. 
\end{cor}

Since we assume $B\subseteq \overline{\mathrm{GL}(B^{\sim})}$, we obtain the following proposition 
by the same argument as in the proof of \cite[Proposition 4.9]{Na3}. 

\begin{pro}\label{pro:lift-unitary}
Let $u$ be a unitary element in $F(\Phi(\mathcal{D}), B)$. 
Then there exists a unitary element $w$ in 
$(B^{\sim})^{\omega}\cap\Phi (\mathcal{D})^{\prime}$ such that $u= [w]$. 
\end{pro}

There exists a homomorphism $\rho$ from $F(\Phi (\mathcal{D}), B)\otimes \mathcal{D}$ to 
$B^{\omega}$ such that 
$$
\rho ([(x_n)_n]\otimes a)=(x_n\Phi_n (a))_n
$$ 
for any $[(x_n)_n]\in F(\Phi(\mathcal{D}), B)$ and $a\in\mathcal{D}$. 
For a projection $p$ in $F(\Phi(\mathcal{D}), B)$, 
put $B_{p}^{\omega}:= \overline{\rho (p\otimes s)B^{\omega}\rho (p\otimes s)}$ where 
$s$ is a strictly positive element in $\mathcal{D}$. 
Define a homomorphism $\sigma_p$ from $\mathcal{D}$ to $B_{p}^{\omega}$ by 
$\sigma_{p}(a):=\rho (p\otimes a)$ for any $a\in\mathcal{D}$. 
Since $B$ has strict comparison, 
we see that if $p$ is a projection in $F(\Phi(\mathcal{D}), B)$ such that $\tau_{B, \omega}(p)>0$, 
then $\sigma_p$ is $(L,N)$-full for some maps $L$ and $N$ by 
\cite[Lemma 3.5 and Proposition 3.7]{Na3}. 
(We refer the reader to \cite[Section 3]{Na3} for details of the $(L,N)$-fullness.) 
Therefore \cite[Proposition 3.3]{Na3} implies the following theorem. 
We may regard this theorem as a variant of Elliott-Gong-Lin-Niu's stable uniqueness 
theorem \cite[Corollary 3.15]{EGLN}(see also \cite[Corollary 8.16]{EN}). 
Note that \cite[Proposition 3.3]{Na3} is also based on the results 
in \cite{EllK}, \cite{G}, \cite{DE1} and \cite{DE2}. 

\begin{thm}\label{thm:stable-uniqueness}
Let $\Omega$ be a compact metrizable space. For any finite subsets 
$F_1\subset C(\Omega)$, $F_2\subset \mathcal{D}$ and $\varepsilon>0$, 
there exist finite subsets $G_1\subset C(\Omega)$, $G_2\subset \mathcal{D}$, 
$m\in\mathbb{N}$  and $\delta >0$ such that the following holds. 
Let $p$ be a projection in $F(\Phi(\mathcal{D}), B)$ such that 
$\tau_{B, \omega} (p)>0$. 
For any contractive ($G_1\odot G_2, \delta$)-multiplicative maps 
$\psi_1, \psi_2 : C(\Omega)\otimes \mathcal{D}\to B_p^{\omega}$, 
there exist a unitary element $u$ in $M_{m^2+1}(B_p^{\omega})^{\sim}$ and 
$z_1,z_2,...,z_m\in\Omega$ such that 
\begin{align*}
\| u & (\psi_1(f\otimes b) \oplus  \overbrace{\bigoplus_{k=1}^m f(z_k)\rho (p\otimes b)\oplus \cdots \oplus\bigoplus_{k=1}^m f(z_k)\rho (p\otimes b) }^m) u^* \\
& - \psi_2(f\otimes b)\oplus \overbrace{\bigoplus_{k=1}^m f(z_k)\rho (p\otimes b) \oplus \cdots \oplus \bigoplus_{k=1}^m f(z_k)\rho (p\otimes b)}^m\| < \varepsilon 
\end{align*}
for any $f\in F_1$ and $b\in F_2$. 
\end{thm}

Using Proposition \ref{pro:central-strict-comparison}, Proposition \ref{pro:property-w-ii} 
and Corollary \ref{cor:MvN-u} instead of \cite[Proposition 4.1]{Na3}, \cite[Proposition 4.2]{Na3} 
and \cite[Proposition 4.8]{Na3}, the same proof as \cite[Lemma 5.1]{Na3} shows the 
following lemma. 

\begin{lem}\label{lem:approximation}
Let $\Omega$ be a compact metrizable space, and let $F$ be a finite subset of $C(\Omega)$ 
and $\varepsilon >0$. 
Suppose that $\psi_1$ and $\psi_2$ are unital homomorphisms from $C(\Omega)$ to 
$F(\Phi (\mathcal{D}), B)$ 
such that 
$
\tau_{B, \omega} \circ \psi_1  = \tau_{B, \omega} \circ \psi_2 .
$ 
Then there exist a projection $p\in F(\Phi (\mathcal{D}), B)$, 
$(F,\varepsilon)$-multiplicative unital c.p. maps $\psi_1^{\prime}$ and $\psi_2^{\prime}$ from 
$C(\Omega)$ to $pF(\Phi (\mathcal{D}), B)p$, a unital homomorphism $\sigma$ 
from $C(\Omega)$ to $(1-p)F(\Phi (\mathcal{D}), B)(1-p)$ with finite-dimensional 
range and a unitary element $u\in F(\Phi (\mathcal{D}), B)$ 
such that 
$$
0 <\tau_{B, \omega} (p) < \varepsilon, \;
\| \psi_1 (f)- (\psi_1^{\prime}(f)+ \sigma (f))\| <\varepsilon,  \;
\| \psi_2 (f)- u(\psi_2^{\prime}(f)+ \sigma (f))u^*\| <\varepsilon 
$$
for any $f\in F$. 
\end{lem}

The following lemma is essentially the same as \cite[Theorem 5.2]{Na3} and \cite[Theorem 5.2]{Na}. 

\begin{lem}\label{lem:unitary-equivalence-1}
Let $\Omega$ be a compact metrizable space, and let $F_1$ be a finite subset of 
$C(\Omega)$ and $F_2$ a finite subset of $\mathcal{D}$, and let $\varepsilon >0$.  
Then there exist mutually orthogonal positive elements $h_1,h_2,...,h_{l}$ in $C(\Omega)$ of 
norm one such that the following holds. 
If $\psi_1$ and $\psi_2$ are  unital homomorphisms from $C(\Omega)$ to 
$F(\Phi(\mathcal{D}), B)$ such that  
$$
\tau_{B, \omega} (\psi_1 (h_i))>0, \; 1 \leq \forall i \leq l
\quad \text{and} \quad  
\tau_{B, \omega} \circ \psi_1  =\tau_{B, \omega} \circ \psi_2,  
$$
then there exist a unitary elements $u$ in $(B^{\omega})^{\sim}$ such that 
$$
\| u\rho (\psi_1 (f)\otimes a)u^* - \rho (\psi_2(f)\otimes  a) \| < \varepsilon 
$$
for any $f\in F_1$, $a\in F_2$. 
\end{lem}
\begin{proof}
Take positive elements $h_1,h_2,...,h_l$ in $C(\Omega)$ by the same way as  in the proof of 
\cite[Theorem 5.2]{Na3}. 
Let $\psi_1$ and $\psi_2$ be 
unital homomorphisms from $C(\Omega)$ to 
$F(\Phi(\mathcal{D}), B)$ such that  
$\tau_{B, \omega} (\psi_1 (h_i))>0$ for any  $1 \leq i \leq l$ and 
$\tau_{B, \omega} \circ \psi_1=\tau_{B, \omega} \circ \psi_2$. 
Define homomorphisms $\Psi_1$ and $\Psi_2$ 
from $C(\Omega)\otimes \mathcal{D}$ to $B^{\omega}$ by 
$$
\Psi_1: =\rho\circ (\psi_1\otimes\mathrm{id}_{\mathcal{D}}) \quad \text{and} 
\quad \Psi_2:=\rho \circ (\psi_2\otimes \mathrm{id}_{\mathcal{D}}).
$$ 
Note that there exists $\nu >0$ such that $\tau_{B, \omega}(\psi_1(h_{i}))\geq \nu$ for any  
$1 \leq i \leq l$. 
Using Proposition \ref{pro:lift-unitary}, Theorem \ref{thm:stable-uniqueness} and 
Lemma \ref{lem:approximation} instead of \cite[Corollary 4.10]{Na3}, \cite[Corollary 3.8]{Na3} 
and \cite[Lemma 5.1]{Na3}, the same argument as in the proof of \cite[Theorem 5.2]{Na3} shows that 
there exist a unitary elements $u$ in $(B^{\omega})^{\sim}$ such that 
$$
\| u\Psi_1 (f\otimes a)u^* - \Psi_2(f\otimes  a) \| < \varepsilon 
$$
for any $f\in F_1$, $a\in F_2$. 
Therefore we obtain the conclusion. 
\end{proof}

The following theorem is a generalization of \cite[Theorem 5.3]{Na3}. See also 
\cite[Theorem 5.3]{Na}.

\begin{thm}
Let $N_1$ and $N_2$ be normal elements in $F(\Phi (\mathcal{D}), B)$ such that 
$\mathrm{Sp} (N_1)=\mathrm{Sp} (N_2)$ and 
$\tau_{B, \omega} (f(N_1)) >0$ for any $f\in C(\mathrm{Sp}(N_1))_{+}\setminus \{0\}$. 
Then there exists a unitary element $u$ in $F(\Phi(\mathcal{D}), B)$ 
such that $uN_1u^* =N_2$ if and only if 
$
\tau_{B, \omega} (f(N_1))= \tau_{B, \omega} (f(N_2))
$ 
for any $f\in C(\mathrm{Sp}(N_1))$. 
\end{thm}
\begin{proof}
It is enough to show the if part because the only if part is obvious. 
Let $\Omega:= \mathrm{Sp}(N_1)=\mathrm{Sp}(N_2)$, and 
define unital homomorphisms $\psi_1$ and $\psi_2$ from $C(\Omega)$ to 
$F(\Phi (\mathcal{D}), B)$ by $\psi_1(f):=f(N_1)$ and $\psi_2(f):=f(N_2)$ for any 
$f\in C(\Omega)$. 
By the Choi-Effros lifting theorem, there exist sequences of unital c.p. maps 
$\{\psi_{1,n}\}_{n\in\mathbb{N}}$ and 
$\{\psi_{2, n}\}_{n\in\mathbb{N}}$ from $C(\Omega)$ to $B^{\sim}$ such that 
$\psi_1(f)=[(\psi_{1, n}(f))_n]$ and $\psi_2(f)=[(\psi_{2,n}(f))_n]$ for any $f\in C(\Omega)$. 
Let $F_1:=\{1, \iota \}\subset C(\Omega)$ where $\iota$ is the identity function on $\Omega$, 
that is $\iota(z)=z$ for any $z\in \Omega$, and let $\{F_{2,m}\}_{m\in\mathbb{N}}$ be an increasing 
sequence of finite subsets in $\mathcal{D}$ such that $\mathcal{D}=
\overline{\bigcup_{m\in\mathbb{N}}F_{2,m}}$. 
For any $m\in\mathbb{N}$, applying Lemma \ref{lem:unitary-equivalence-1} to $F_1$, 
$F_{2,m}$ and $1/m$, we obtain mutually orthogonal positive elements $h_{1,m}$, $h_{2,m}$,...,$h_{l(m),m}$ 
in $C(\Omega)$ of norm one. Since we have 
$$
\tau_{B, \omega} (\psi_1 (h_{i,m}))>0, \; 1 \leq \forall i \leq l(m)
\quad \text{and} \quad  
\tau_{B, \omega} \circ \psi_1  =\tau_{B, \omega} \circ \psi_2  
$$
by the assumption, Lemma \ref{lem:unitary-equivalence-1} implies that there exists a unitary 
element $(u_{m,n})_n$ in $(B^{\omega})^{\sim}$ such that 
$$
\| (u_{m,n})_n\rho (\psi_1(f)\otimes a)(u_{m,n}^*)_n - \rho (\psi_2(f)\otimes  a) \| < \frac{1}{m} 
$$
for any $f\in F_1$, $a\in F_{2,m}$. By the definition of $\rho$, we have 
$$
\lim_{n\to\omega} \| u_{m,n}\psi_{1,n}(f)\Phi_n(a)u_{m,n}^*- \psi_{2, n}(f)\Phi_n(a) \| < \frac{1}{m}
$$
for any $f\in F_1$, $a\in F_{2,m}$. Therefore we inductively obtain a decreasing sequence 
$\{X_{m}\}_{m\in\mathbb{N}}$ of elements in $\omega$ such that for any $n\in X_{m}$, 
$$
\| u_{m,n}\psi_{1,n}(f)\Phi_n(a)u_{m,n}^*- \psi_{2, n}(f)\Phi_n(a) \| < \frac{1}{m}
$$
for any $f\in F_1$, $a\in F_{2,m}$.
Set 
$$
u_{n} := \left\{\begin{array}{cl}
1 & \text{if } n\notin X_1   \\
u_{m,n} & \text{if } n\in X_m\setminus X_{m+1}\quad (m\in\mathbb{N})
\end{array}
\right..
$$
Then we have 
$$
\lim_{n\to\omega} \| u_n\Phi_n(a)u_n^*- \Phi_n(a)\|=0, \quad 
\lim_{n\to\omega} \| u_n\psi_{1, n}(\iota)\Phi_n(a)u_n^*- \psi_{2, n}(\iota)\Phi_n(a)\|=0
$$
for any $a\in\mathcal{D}$. 
Therefore, $(u_n)_n\in (B^{\sim})^{\omega}\cap\Phi(\mathcal{D})^{\prime}$ and 
$[(u_n)_n]N_1[(u_n)_n]^*= N_2$ in $F(\Phi(\mathcal{D}), B)$. Since $[(u_n)_n]$ is a unitary element 
in $F(\Phi(\mathcal{D}), B)$, we obtain the conclusion. 
\end{proof}

The following corollary is an immediate consequence of the theorem above. 

\begin{cor} (cf. \cite[Corollary 5.4]{Na})
Let $p$ and $q$ be projections in $F(\Phi (\mathcal{D}), B)$ such that 
$0< \tau_{B \omega} (p) <1$.  Then $p$ and $q$ are unitarily equivalent if and only if 
$
\tau_{B, \omega} (p)= \tau_{B, \omega} (q)
$. 
\end{cor}

The corollary above and the same argument as in the proof of \cite[Corollary 5.5]{Na} show 
the following theorem. 

\begin{thm}\label{thm:relative-property-w}
Let $p$ and $q$ be projections in $F(\Phi (\mathcal{D}), B)$ such that 
$0< \tau_{B, \omega} (p) \leq 1$.  Then $p$ and $q$ are Murray-von Neumann equivalent if and only if 
$
\tau_{B, \omega} (p)= \tau_{B, \omega} (q)
$. 
\end{thm}

By Proposition \ref{pro:property-w-ii} and applying the theorem above 
to $B=\mathcal{D}$ and $\Phi=\mathrm{id}_{\mathcal{D}}$, we obtain the following corollary. 

\begin{cor}\label{cor:property-w}
Let $\mathcal{D}$ be a simple separable nuclear 
monotracial $M_{2^{\infty}}$-stable C$^*$-algebra which is $KK$-equivalent to $\{0\}$. 
Then $\mathcal{D}$ has property W. 
\end{cor}

\section{Uniqueness theorem}\label{sec:uniqueness}

In this section we shall show that if $D$ has property W, then every trace preserving 
endomorphism of $D$ is approximately inner. 
Furthermore, we shall consider a uniqueness theorem for approximate homomorphisms 
from a simple separable nuclear monotracial $M_{2^{\infty}}$-stable C$^*$-algebra 
$\mathcal{D}$ which is $KK$-equivalent to $\{0\}$ for an existence theorem in 
Section \ref{sec:existence}. 

Let $D$ be a simple separable nuclear monotracial C$^*$-algebra, and let 
$\varphi$ be a trace preserving endomorphism of $D$. 
Define a homomorphism $\Phi$ from $D$ to $M_{2}(D)$ by 
$$
\Phi (a):= \left(\begin{array}{cc}
               a   &   0    \\ 
               0   &  \varphi (a)    
 \end{array} \right)
$$
for any $a\in D$. Since $\varphi$ is trace preserving, we see that $\tau_{M_2(D), \omega}|_{\Phi(D)}$ 
is a state. Hence $\tau_{M_{2}(D), \omega}$ is a tracial state on 
$F(\Phi(D), M_2(D))$. (See Proposition \ref{pro:trace-induce}.) 
Define homomorphisms $\iota_{11}$ and $\iota_{22}$ from $F(D)$ to $F(\Phi(D), M_2(D))$ by 
$$
\iota_{11}([(x_n)_n]):=\left[\left(\left(\begin{array}{cc}
             x_n   &   0    \\ 
             0     &   0    
 \end{array} \right)\right)_{n}\right]
\quad \text{and} \quad 
\iota_{22}([(x_n)_n]):=\left[\left(\left(\begin{array}{cc}
             0   &   0               \\ 
             0   &  \varphi(x_n)    
 \end{array} \right)\right)_{n}\right]
$$ 
for any $[(x_n)_n]$ in $F(D)$. It is easy to see that $\iota_{11}$ and $\iota_{22}$ are well-defined.
Put $p:=\iota_{11}(1)$ and $q:=\iota_{22}(1)$. 
Note that $p$ and $q$ are projections in $F(\Phi(D), M_2(D))$ and 
if $\{h_n\}_{n\in\mathbb{N}}$ is an approximate unit for $D$, then 
$$
p=\left[\left(\left(\begin{array}{cc}
             h_n   &   0    \\ 
             0     &   0    
 \end{array} \right)\right)_{n}\right]
\quad \text{and} \quad 
q=\left[\left(\left(\begin{array}{cc}
             0   &   0               \\ 
             0   &  \varphi(h_n)    
 \end{array} \right)\right)_{n}\right].
$$  
It can be easily checked that $\iota_{11}$ is an isomorphism from $F(D)$ onto 
$pF(\Phi(D), M_2(D))p$. 

\begin{lem}\label{lem:w-car-stable}
Let $D$ be a  a simple separable nuclear monotracial C$^*$-algebra with property W. 
Then $D$ is $M_{2^{\infty}}$-stable, and hence $D$ is $\mathcal{Z}$-stable.  
\end{lem}
\begin{proof}
Since $D$ has property W, there exists a projection $p$ in $F(D)$ such that 
$\tau_{D,\omega}(p)=1/2$. Moreover, $p$ is Murray-von Neumann equivalent to $1-p$. 
Hence there exists a unital homomorphism from $M_{2}(\mathbb{C})$ to 
$F(D)$. By \cite[Corollary 1.13]{Kir2} and \cite[Proposition 4.11]{Kir2} 
(see \cite[Proposition 2.12]{BKR} for the pioneering work),  $D$ is $M_{2^{\infty}}$-stable. 
\end{proof}

The lemma above implies that if $D$ has property W, then $D$ has strict comparison and 
$D\subseteq \overline{\mathrm{GL}(D^{\sim})}$ by \cite{Ror} and \cite{Rob2}. 
Furthermore, $F(\Phi(D), M_2(D))$ is monotracial and has strict comparison by 
Proposition \ref{pro:central-strict-comparison}. 
The following lemma is related to \cite[Lemma 6.2]{Na}.

\begin{lem}\label{lem:matrix-trick}
With notation as above, if $D$ has property W, then $p$ is Murray-von Neumann equivalent to 
$q$ in $F(\Phi(D), M_2(D))$. 
\end{lem}
\begin{proof}
For any $m\in\mathbb{N}$, there exists a projection $q_{m}$ in $F(D)$ such that 
$\tau_{D, \omega}(q_m)=1-1/m$ because $D$ has property W. 
Proposition \ref{pro:central-strict-comparison} implies that 
there exists a contraction $r_m$ in $F(\Phi (D), M_2(D))$ such that $r_m^*pr_m=\iota_{22}(q_m)$. 
By a diagonal argument, we see that there exist a projection $q^{\prime}$ in $F(D)$ and 
a contraction $r$ in $F(\Phi(D), M_2(D))$ such that $\tau_{D, \omega}(q^{\prime})=1$ and 
$r^*pr=\iota_{22}(q^{\prime})$. Note that $\iota_{22}(q^{\prime})$ is Murray-von Neumann equivalent 
to $prr^*p$.  There exists a projection $p^{\prime}$ in $F(D)$ such that 
$\iota_{11}(p^{\prime})=prr^*p$ and $\tau_{D, \omega}(p^{\prime})=1$ because 
$\iota_{11}$ is an isomorphism from $F(D)$ onto $pF(\Phi(D), M_2(D))p$. 
Since $D$ has property W, there exist $v_1$ and $v_2$ in $F(D)$ such that $v_1^*v_1=1$, 
$v_1v_1^*=p^{\prime}$, $v_2^*v_2=1$ and $v_2v_2^*=q^{\prime}$. 
Therefore we have 
$$
p=\iota_{11}(1) \sim \iota_{11}(p^{\prime})=prr^*p\sim r^*pr
= \iota_{22}(q^{\prime}) \sim  \iota_{22}(1)=q. 
$$
\end{proof}

The following theorem is one of the main theorems in this section. 

\begin{thm}\label{thm:uniqueness-i}
Let $D$ be a simple separable nuclear monotracial C$^*$-algebra with property W, and let 
$\varphi$ be a trace preserving endomorphism of $D$. Then $\varphi$ is approximately inner. 
\end{thm}
\begin{proof}
By Lemma \ref{lem:matrix-trick}, there exists a contraction $V$ in $F(\Phi(D), M_2(D))$ such that 
$$
V^*V=\left[\left(\left(\begin{array}{cc}
            h_n   &   0    \\ 
             0     &   0    
 \end{array} \right)\right)_{n}\right]
\quad \text{and} \quad 
VV^*=\left[\left(\left(\begin{array}{cc}
             0   &   0               \\ 
             0   &  \varphi(h_n)    
 \end{array} \right)\right)_{n}\right]
$$
where $\{h_n\}_{n\in\mathbb{N}}$ is an approximate unit for $D$. 
It can be easily checked that there exists an element $(v_n)_n$ in $D^{\omega}$ such that 
$$
V= \left[\left(\left(\begin{array}{cc}
             0       &   0    \\ 
            v_n     &   0    
 \end{array} \right)\right)_{n}\right],
$$
and we have 
$$
(v_nx)_n=(\varphi (x)v_n)_n,  \quad (v_n^*v_nx)_n=x \quad \text{and} \quad 
(v_nv_n^*\varphi(x))_n=\varphi(x)
$$
for any $x\in D$. 
Since $(v_nx)_n=(\varphi (x)v_n)_n$ and $(\varphi(x)v_nv_n^*)_n=\varphi(x)$, 
we have $(v_nxv_n^*)_n= \varphi (x)$ for any $x\in D$. 
Because of $D\subseteq \overline{\mathrm{GL}(D^{\sim})}$, we may assume that 
$v_n$ is an invertible element in $D^{\sim}$ for any $n\in\mathbb{N}$. (See the proof of 
Proposition \ref{pro:non-sep-gl}.)
For any $n\in\mathbb{N}$, let $u_n:= v_n(v_n^*v_n)^{-1/2}$. 
Then $u_n$ is a unitary element in $D^{\sim}$. 
Since $ (v_n^*v_nx)_n=x$, we have $(u_nx)_n=(v_n(v_n^*v_n)^{-1/2}x)_n=(v_nx)_n$ for any $x\in D$. 
Therefore 
$$
\varphi(x)= (v_nxv_n^*)_n= (u_nxv_n^*)_n =(u_n(v_nx^*)^*)_n=(u_n(u_nx^*)^*)_n=(u_nxu_n^*)_n
$$
for any $x\in D$. Consequently, $\varphi$ is approximately inner. 
\end{proof}

Let $\mathcal{D}$ be a simple separable nuclear monotracial $M_{2^{\infty}}$-stable 
C$^*$-algebra which is $KK$-equivalent to $\{0\}$.
In the rest of this section, we shall consider a uniqueness theorem for 
approximate homomorphisms from $\mathcal{D}$ to certain C$^*$-algebras. 
Let $B$ be a simple monotracial C$^*$-algebra with strict comparison,  
$B \subseteq \overline{\mathrm{GL}(B^{\sim})}$ and  
$M_2(B)\subseteq \overline{\mathrm{GL}(M_2(B)^{\sim})}$, and let $\varphi$ and $\psi$ 
be homomorphisms from $\mathcal{D}$ to $B^{\omega}$ such that 
$\tau_{\mathcal{D}}=\tau_{B, \omega}\circ \varphi=\tau_{B, \omega}\circ \psi$. 
By the Choi-Effros lifting theorem, there exist sequences of contractive c.p. maps 
$\varphi_n$ and $\psi_n$ from $\mathcal{D}$ to $B$ such that 
$\varphi(a)= (\varphi_n(a))_n$ and $\psi (a)=(\psi_n(a))_n$ for any $a\in\mathcal{D}$. 
Define a homomorphism $\Phi$ from $\mathcal{D}$ to $M_{2}(B)^{\omega}$ by 
$$
\Phi (a):= \left(\left(\begin{array}{cc}
      \varphi_n(a)   &   0    \\ 
               0   &  \psi_n (a)    
 \end{array} \right)\right)_n
$$
for any $a\in \mathcal{D}$. 
Since $\tau_{\mathcal{D}}=\tau_{B, \omega}\circ \varphi=\tau_{B, \omega}\circ \psi$, 
$\tau_{M_2(B), \omega}|_{\Phi(\mathcal{D})}$ is a state. Hence $\tau_{M_{2}(B), \omega}$ is a 
tracial state on $F(\Phi(\mathcal{D}), M_2(B))$ as above. 
Since $\mathcal{D}$ is separable, there exist elements $(s_n)_n$ and $(t_n)_n$ in 
$B^{\omega}$ such that  $[(s_n)_n]=1$ in $F(\varphi(\mathcal{D}), B)$ and $[(t_n)_n]=1$ in 
$F(\psi(\mathcal{D}), B)$ by arguments in Section \ref{sec:pre-kir}.  
Put
$$
p:=\left[\left(\left(\begin{array}{cc}
             s_n   &   0    \\ 
             0     &   0    
 \end{array} \right)\right)_{n}\right]
\quad \text{and} \quad 
q:=\left[\left(\left(\begin{array}{cc}
             0   &   0               \\ 
             0   &  t_n    
 \end{array} \right)\right)_{n}\right]
$$
in $F(\Phi(\mathcal{D}), M_2(B))$. It is easy to see that $p$ and $q$ are projections in 
$F(\Phi(\mathcal{D}), M_2(B))$ such that $\tau_{M_2(B), \omega}(p)=\tau_{M_2(B), \omega}(q)=1/2$. 
Theorem \ref{thm:relative-property-w} implies that $p$ is Murray-von Neumann equivalent to 
$q$. 
Therefore we obtain the following theorem by a similar argument as in the proof of 
Theorem \ref{thm:uniqueness-i}.

\begin{thm}\label{thm:uniqueness-ii}
Let $\mathcal{D}$ be a simple separable nuclear monotracial $M_{2^{\infty}}$-stable 
C$^*$-algebra which is $KK$-equivalent to $\{0\}$ and  $B$ a simple monotracial C$^*$-algebra 
with strict comparison, $B\subseteq \overline{\mathrm{GL}(B^{\sim})}$ and 
$M_2(B)\subseteq \overline{\mathrm{GL}(M_2(B)^{\sim})}$. 
If $\varphi$ and $\psi$ are homomorphisms from $\mathcal{D}$ to $B^{\omega}$ 
such that $\tau_{\mathcal{D}}=\tau_{B, \omega}\circ \varphi=\tau_{B, \omega}\circ \psi$, 
then there exists a unitary element $u$ in $(B^{\sim})^{\omega}$ such that 
$\varphi(a)= u\psi(a)u^*$ for any $a\in\mathcal{D}$. 
\end{thm}

The following corollary is an immediate consequence of the theorem above. 

\begin{cor}
Let $\mathcal{D}$ be a simple separable nuclear monotracial $M_{2^{\infty}}$-stable 
C$^*$-algebra which is $KK$-equivalent to $\{0\}$ and  $B$ a simple monotracial C$^*$-algebra 
with strict comparison, $B\subseteq \overline{\mathrm{GL}(B^{\sim})}$ 
and $M_2(B)\subseteq \overline{\mathrm{GL}(M_2(B)^{\sim})}$. 
If $\varphi$ and $\psi$ are trace preserving homomorphisms from $\mathcal{D}$ to $B$, 
then $\varphi$ is approximately unitarily equivalent to $\psi$.
\end{cor}

\begin{rem}
If $B$ is a simple separable exact monotracial $\mathcal{Z}$-stable C$^*$-algebra, then 
$B$ has strict comparison, $B\subseteq \overline{\mathrm{GL}(B^{\sim})}$ and 
$M_2(B)\subseteq \overline{\mathrm{GL}(M_2(B)^{\sim})}$ by \cite{Ror} and \cite{Rob2}. 
\end{rem}

The following corollary is also an immediate consequence of Theorem \ref{thm:uniqueness-ii}. 

\begin{cor}\label{cor:uniqueness-iii}
Let $\mathcal{D}$ be a simple separable nuclear monotracial $M_{2^{\infty}}$-stable 
C$^*$-algebra which is $KK$-equivalent to $\{0\}$ and  $B$ a simple monotracial C$^*$-algebra 
with strict comparison, $B\subseteq \overline{\mathrm{GL}(B^{\sim})}$ 
and $M_2(B)\subseteq \overline{\mathrm{GL}(M_2(B)^{\sim})}$. 
For any finite subset $F\subset \mathcal{D}$ and $\varepsilon>0$, there exist a finite subset 
$G\subset \mathcal{D}$ and $\delta>0$ such that the following holds. 
If $\varphi$ and $\psi$ are $(G, \delta)$-multiplicative maps from $\mathcal{D}$ to 
$B$ such that 
$$
|\tau_{B}(\varphi (a))-\tau_{\mathcal{D}}(a)|< \delta \quad \text{and} \quad 
|\tau_{B}(\psi (a))-\tau_{\mathcal{D}}(a)|< \delta 
$$
for any $a\in G$, then there exists a unitary element $u$ in $B^{\sim}$ such that 
$$
\| \varphi (a) - u \psi (a) u^*\|< \varepsilon
$$
for any $a\in F$. 
\end{cor}

\section{Existence theorem}\label{sec:existence}

In this section, we assume that $\mathcal{D}$ is a simple separable nuclear monotracial 
$M_{2^{\infty}}$-stable C$^*$-algebra which is $KK$-equivalent to $\{0\}$ and 
$B$ is a simple separable exact monotracial $\mathcal{Z}$-stable C$^*$-algebra. 
We shall show that there exists a trace preserving homomorphism from 
$\mathcal{D}$ to $B$. Many arguments in this section are motivated by 
Schafhauser's proof \cite{Sc} (see also \cite{Sc2}) of 
the Tikuisis-White-Winter theorem \cite{TWW}. 

The following lemma is related to \cite[Lemma 2.2]{KP}. 

\begin{lem}\label{lem:kp}
Let $\mathcal{D}$ be a simple separable nuclear monotracial $M_{2^{\infty}}$-stable 
C$^*$-algebra which is $KK$-equivalent to $\{0\}$ and $B$ a simple separable exact 
monotracial $\mathcal{Z}$-stable C$^*$-algebra. 
If there exists a homomorphism $\varphi$ from $\mathcal{D}$ to $B^{\omega}$ such that 
$\tau_{B, \omega}\circ \varphi = \tau_{\mathcal{D}}$, then there exists a trace preserving 
homomorphism from $\mathcal{D}$ to $B$. 
\end{lem}
\begin{proof}
By the Choi-Effros lifting theorem, there exists a sequence $\{\varphi_n\}_{n\in\mathbb{N}}$ 
of contractive c.p. maps from $\mathcal{D}$ to $B$ such that $\varphi (a)=(\varphi_n(a))_n$ 
for any $a\in\mathcal{D}$. 
Let $\{F_{m}\}_{m\in\mathbb{N}}$ be an increasing 
sequence of finite subsets in $\mathcal{D}$ such that $\mathcal{D}=
\overline{\bigcup_{m\in\mathbb{N}}F_{m}}$. 
For any $m\in\mathbb{N}$, applying Corollary \ref{cor:uniqueness-iii} to 
$F_{m}$ and $1/2^{m}$, we obtain a finite subset $G_m$ of $\mathcal{D}$ and $\delta_{m}>0$. 
We may assume that $G_{m}\subset G_{m+1}$, $\delta_{m}>\delta_{m+1}$ for any 
$m\in\mathbb{N}$ and $\lim_{m\to \infty}\delta_{m}=0$. Since we have 
$$
\lim_{n\to\omega} \| \varphi_n(ab)-\varphi_n(a)\varphi_n(b)\|=0 \quad \text{and} \quad 
\lim_{n\to\omega} |\tau_{B}(\varphi_n(a)) - \tau_{\mathcal{D}}(a)|=0
$$
for any $a,b\in\mathcal{D}$, there exists a subsequence $\{\varphi_{n(m)}\}_{m\in\mathbb{N}}$ 
of $\{\varphi_n\}_{n\in\mathbb{N}}$ such that 
$$
\|\varphi_{n(m)}(ab) -\varphi_{n(m)}(a)\varphi_{n(m)}(b) \|< \delta_{m} \quad \text{and} \quad 
|\tau_{B}(\varphi_{n(m)}(a))-\tau_{\mathcal{D}}(a) | < \delta_{m}
$$
for any $a, b\in G_{m}$. Corollary \ref{cor:uniqueness-iii} implies that for any $m\in\mathbb{N}$, 
there exists a unitary element $u_{m}$ in $B^{\sim}$ such that 
$$
\| \varphi_{n(m)}(a)- u_{m}\varphi_{n(m+1)}(a)u_{m}^* \|< \frac{1}{2^m}
$$
for any $a\in F_m$.  Therefore it can easily be checked that the limit 
$$
\lim_{m\to \infty} u_{1}u_{2}\cdots u_{m-1}\varphi_{n(m)}(a)u_{m-1}^*\cdots u_{2}^*u_{1}^*
$$
exists for any $a\in \mathcal{D}$. 
Define 
$\psi(a):=\lim_{m\to \infty} u_{1}u_{2}\cdots u_{m-1}\varphi_{n(m)}(a)u_{m-1}^*\cdots u_{2}^*u_{1}^*$ 
for any $a\in\mathcal{D}$, then $\psi$ is a trace preserving homomorphism from $\mathcal{D}$ 
to $B$. 
\end{proof}

By the lemma above, it is enough to show that there exists a homomorphism $\varphi$ from 
$\mathcal{D}$ to $B^{\omega}$ such that $\tau_{B, \omega}\circ \varphi = \tau_{\mathcal{D}}$. 
Borrowing Schafhauser's idea in \cite{Sc}, we shall show this. 
By arguments in Section \ref{sec:MS}, there exists the following extension: 
$$
\xymatrix{
\eta: & 0 \ar[r]  & J \ar[r] & B^{\omega} \ar[r]^{\varrho} & M \ar[r] & 0
}
$$
where $M$ is a von Neumann algebraic ultrapower of $\pi_{\tau_B}(B)^{''}$ and 
$J=\mathrm{ker}\; \varrho=\{(x_n)_n\in B^{\omega}\; |\; 
\tilde{\tau}_{B, \omega}((x_n^*x_n)_n)=0 \}$. 
Note that $J$ is known as the trace kernel ideal. 
Also,  
$M$ is a II$_1$ factor because $B$ is infinite-dimensional 
(which is implied by $\mathcal{Z}$-stability) 
and monotracial. Since $\mathcal{D}$ is monotracial and nuclear, 
$\pi_{\tau_{\mathcal{D}}}(\mathcal{D})^{''}$ is the injective II$_1$ factor. 
Hence there exists a unital homomorphism from 
$\pi_{\tau_{\mathcal{D}}}(\mathcal{D})^{''}$ to $M$ (see, for example, \cite[XIV. Proposition 2.15]{Tak}). 
In particular, there exists a trace preserving homomorphism $\Pi$ from $\mathcal{D}$ to $M$. 
Consider the pullback extension 
$$
\xymatrix{
\Pi^*\eta: &   0 \ar[r] 
& J \ar[r]\ar@{=}[d] & E \ar[r]^{\hat{\varrho}}\ar[d]^{\hat{\Pi}} 
& \mathcal{D} \ar[r]\ar[d]^{\Pi} & 0 \\
\eta: & 0 \ar[r]  & J \ar[r] & B^{\omega} \ar[r]^{\varrho} & M \ar[r] & 0
}
$$
where $E=\{(a,x)\in \mathcal{D}\oplus B^{\omega}\; |\; \Pi(a)=\varrho(x)\}$, 
$\hat{\varrho}((a,x))=a$ and $\hat{\Pi}((a,x))=x$ for any $(a,x)\in E$. 
If we could show that $\Pi^*\eta$ is a split extension with a cross section $\gamma$, 
then $\hat{\Pi}\circ \gamma$ is a homomorphism from $\mathcal{D}$ to 
$B^{\omega}$ such that $\tau_{B, \omega}\circ \hat{\Pi}\circ \gamma=\tau_{\mathcal{D}}$. 
But we were unable to show this, immediately. 
Note that we need to consider a separable extension in order 
to use $KK$-theory and some results in \cite{EllK} and \cite{G}. 
We shall construct a suitable separable extension $\eta_0$ by Blackadar's technique 
(see \cite[II.8.5]{Bla}). 

We shall recall some definitions and some results in \cite{EllK} and \cite{G}. 
An extension $0\longrightarrow I \longrightarrow C \longrightarrow A\longrightarrow 0$ 
is said to be \textit{purely large} if for any $x\in C\setminus  I$, 
$\overline{xIx^*}$ contains a stable C$^*$-subalgebra which is full in $I$. Note that 
$\overline{xIx^*}=\overline{xx^*Ixx^*}=I\cap \overline{xCx^*}$. 
By \cite[Theorem 2.1]{G} (see also \cite[Corollary 16]{EllK}), if $A$ is non-unital and 
$I$ is stable, then a separable extension 
$0\longrightarrow I \longrightarrow C \longrightarrow A\longrightarrow 0$
is nuclear absorbing if and only if it is purely large. 

\begin{lem}\label{lem:qd}
With notation as above, suppose that there exist separable C$^*$-subalgebras 
$J_{0}\subset J$, $B_{0}\subset B^{\omega}$ and $M_{0}\subset M$ 
such that $J_0$ is stable, 
$$
\xymatrix{
\eta_0: & 0 \ar[r]  & J_0 \ar[r] & B_0 \ar[r]^{\varrho|_{B_0}} & M_0 \ar[r] & 0
}
$$
is a purely large extension and $\Pi (\mathcal{D})\subset M_0$. Then 
there exists a homomorphism $\varphi$ from $\mathcal{D}$ to $B^{\omega}$ such that 
$\tau_{B, \omega}\circ \varphi=\tau_{\mathcal{D}}$. 
\end{lem}
\begin{proof}
Consider the pullback extension 
$$
\xymatrix{
\Pi^*\eta_0: &   0 \ar[r] 
& J_0 \ar[r]\ar@{=}[d] & E_0 \ar[r]^{\hat{\varrho}}\ar[d]^{\hat{\Pi}} 
& \mathcal{D} \ar[r]\ar[d]^{\Pi} & 0 \\
\eta_0: & 0 \ar[r]  & J_0 \ar[r] & B_0 \ar[r]^{\varrho} & M_0 \ar[r] & 0
}
$$
where $E_0=\{(a,x)\in \mathcal{D}\oplus B_0\; |\; \Pi(a)=\varrho(x)\}$, 
$\hat{\varrho}((a,x))=a$ and $\hat{\Pi}((a,x))=x$ for any $(a,x)\in E_0$. 
Since $\eta_0$ is purely large, it can be easily checked that $\Pi^*\eta_0$ is purely large. 
Hence $\Pi^*\eta_0$ is nuclear absorbing by \cite[Theorem 2.1]{G}. 
Because $\mathcal{D}$ is $KK$-equivalent to $\{0\}$ and nuclear, we have 
$\mathrm{Ext}(\mathcal{D}, J_0)=\{0\}$, and hence $[\Pi^*\eta_0]=0$ in 
$\mathrm{Ext}(\mathcal{D}, J_0)$. 
Therefore there exists a (nuclear) split extension $\eta^{\prime}$ such that 
$\Pi^*\eta_0\oplus \eta^{\prime}$ is a split extension. 
Since $\Pi^*\eta_0$ is nuclear absorbing, $\Pi^*\eta_0$ is strongly unitarily 
equivalent to $\Pi^*\eta_0\oplus \eta^{\prime}$, and hence  $\Pi^*\eta_0$ is a split extension. 
Let $\gamma_0$ be a cross section of  $\Pi^*\eta_0$, and define 
$\varphi:=\hat{\Pi}\circ \gamma_0$. Then $\varphi$ is the desired homomorphism. 
\end{proof}

A key result in the proof of purely largeness is the following 
Hjelmborg and R\o rdam's characterization of stable C$^*$-algebras in \cite{HR} and \cite{Ror2}. 

\begin{thm} \label{thm:HR} (Hjelmborg-R\o rdam cf. \cite[Theorem 2.2]{Ror2}) \ \\
Let $A$ be a $\sigma$-unital C$^*$-algebra. Then $A$ is stable if and only if 
for any $a\in A_{+}$ and $\varepsilon>0$, there exist positive elements $a^{\prime}$ and $c$ in 
$A$ such that $\| a-a^{\prime}\|\leq \varepsilon$, $a^{\prime}\sim c$ and 
$\| ac\|\leq \varepsilon$.  
\end{thm}

Before we construct a separable extension $\eta_0$, we shall consider properties of $\eta$. 

\begin{pro}\label{pro:non-sep1} 
With notation as above, let $b$ be a positive element in $B^{\omega}\setminus J$. \ \\
(i) For any positive element $a$ in $\overline{bJb}$, there exists a positive 
element $c$ in $\overline{bJb}$ such that $a\sim c$ and $ac=0$. \ \\
(ii) For any positive element $a$ in $J$ and $\varepsilon>0$, there exist a positive element $d$ 
in $\overline{bJb}$ and an element $r$ in $J$ such that $\| r^*dr-a\|<\varepsilon$. 
\ \\
(iii) For any element $x$ in $B^{\omega}$ and $\varepsilon>0$, there exists an element 
$y$ in $\mathrm{GL}((B^{\omega})^{\sim})$ such that $\| x-y\|< \varepsilon$. 
\end{pro}

For the proof of the proposition above, we need some lemmas. For a positive element 
$a\in A$ and $\varepsilon >0$, we denote by $(a-\varepsilon)_{+}$ the element $f(a)$ in $A$
where $f(t)=\max\{0,t-\varepsilon \}, t\in \mathrm{Sp}(a)$. 
The same proof as in \cite[Proposition 2.4]{Ror3} shows the following lemma. 
See also \cite[Corollary 8]{Ped1}. 

\begin{lem}\label{lem:Pedersen-Rordam}
Let $A$ be a C$^*$-algebra with $A\subseteq \overline{\mathrm{GL}(A^{\sim})}$, 
and let $a$ and $b$ be positive elements in $A$. Then $a$ is Cuntz smaller than $b$ if and only if  
for any $\varepsilon>0$, there exists  a unitary element $u$ in $A^{\sim}$ such that 
$u(a-\varepsilon)_{+}u^*\in \overline{bAb}$. 
\end{lem}

The following lemma can be regarded as an application of the construction of $\mathcal{Z}$. 

\begin{lem}\label{lem:jiang-su-cor}
Let $A$ be a monotracial $\mathcal{Z}$-stable C$^*$-algebra. 
For any $\theta\in(0,1/2)$, there exist positive elements $d$ and $d^{\prime}$ in $A$ 
such that $dd^{\prime}=0$ and $d_{\tau_{A}}((d-\varepsilon)_{+})=
d_{\tau_{A}}((d^{\prime}-\varepsilon)_+)=(1-\varepsilon)\theta$ for any $0\leq \varepsilon \leq 1$. 
\end{lem}
\begin{proof}
Let $\mu$ be the Lebesgue measure on $[0,1]$, and define a tracial state $\tau_{0}$ on 
$C([0,1])$ by $\tau_{0}(f):=\int_{[0,1]} f \; d\mu$ for any $f\in C([0,1])$. 
By \cite[Theorem 2.1.(i)]{Ror}, there exists a unital homomorphism $\psi$ from $C([0,1])$ 
to $\mathcal{Z}$ such that $\tau_0=\tau_{\mathcal{Z}}\circ \psi$. 
Define $f$ and $g$ in $C([0,1])$ by 
$$
f(t):=\left\{\begin{array}{cl}
\frac{2}{\theta}t &\text{if}\quad t\in [0,\frac{\theta}{2}]  \\
-\frac{2}{\theta}t+2 & \text{if}\quad t\in (\frac{\theta}{2},\theta] \\
0 & \text{if}\quad t\in (\theta, 1]
\end{array}
\right.
\quad \text{and} \quad
g(t):=\left\{\begin{array}{cl}
0 &\text{if}\quad t\in [0,\theta]  \\
\frac{2}{\theta}t-2 & \text{if}\quad t\in (\theta, \frac{3\theta}{2}] \\
-\frac{2}{\theta}t+4 & \text{if}\quad t\in (\frac{3\theta}{2}, 2\theta ] \\
0  & \text{if}\quad t\in (2\theta, 1] 
\end{array}
\right..
$$
Note that for any $0\leq \varepsilon \leq 1$, we have 
$$
(f-\varepsilon)_{+}(t)=\left\{\begin{array}{cl}
0 &\text{if}\quad t\in [0,\frac{\varepsilon\theta}{2}]  \\
\frac{2}{\theta}t -\varepsilon & \text{if}\quad t\in (\frac{\varepsilon\theta}{2},\frac{\theta}{2}] \\
-\frac{2}{\theta}t+2-\varepsilon  & \text{if}\quad t\in 
(\frac{\theta}{2}, \theta-\frac{\varepsilon\theta}{2}] \\
0 & \text{if}\quad t\in (\theta-\frac{\varepsilon\theta}{2},1]
\end{array}
\right.
$$
and 
$$
(g-\varepsilon)_{+}(t)=\left\{\begin{array}{cl}
0 &\text{if}\quad t\in [0,\theta+\frac{\varepsilon\theta}{2}]  \\
\frac{2}{\theta}t-2-\varepsilon & \text{if}\quad 
t\in (\theta+\frac{\varepsilon\theta}{2}, \frac{3\theta}{2}] \\
-\frac{2}{\theta}t+4-\varepsilon & \text{if}\quad 
t\in (\frac{3\theta}{2},  2\theta-\frac{\varepsilon\theta}{2}] \\
0  & \text{if}\quad t\in (2\theta-\frac{\varepsilon\theta}{2}, 1] 
\end{array}
\right..
$$
Hence $d_{\tau_0}((f-\varepsilon)_{+})=d_{\tau_0}((g-\varepsilon)_{+})=(1-\varepsilon)\theta$. 
Let $s$ be a strictly positive element in $A$, and put
$$
d:= s\otimes \psi (f) \quad \text{and} \quad d^{\prime}:= s\otimes \psi (g)
$$
in $A\otimes\mathcal{Z}\cong A$. Then $d$ and $d^{\prime}$ are desired positive elements in $A$. 
\end{proof}

\begin{lem}\label{lem:positive-e}
Let $A$ be a simple separable exact monotracial $\mathcal{Z}$-stable C$^*$-algebra, and let 
$b$ be a (non-zero) positive element in $A$. 
For any $\theta\in (0, d_{\tau_{A}}(b)/2)$, there exist positive elements $e$ and $e^{\prime}$ 
in $\overline{bAb}$ such that $ee^{\prime}=0$ and $d_{\tau_{A}}(e)=d_{\tau_{A}}(e^{\prime})>\theta$. 
\end{lem}
\begin{proof}
By Lemma \ref{lem:jiang-su-cor}, there exist contractions $d$ and $d^{\prime}$ in $A$  such that 
$dd^{\prime}=0$ and $\theta < d_{\tau_{A}}(d)=d_{\tau_{A}}(d^{\prime})<d_{\tau_{A}}(b)/2$. 
Furthermore, we may assume that there exists $\varepsilon >0$ such that 
$d_{\tau_{A}}((d-\varepsilon)_{+})=d_{\tau_{A}}((d^{\prime}-\varepsilon)_+)>\theta$. 
Since $A$ has strict comparison and 
$
d_{\tau_{A}}(d+d^{\prime})=d_{\tau_{A}}(d)+d_{\tau_{A}}(d^{\prime})< d_{\tau_{A}}(b)
$, 
Lemma \ref{lem:Pedersen-Rordam} implies that there exists a unitary element $u$ in $A^{\sim}$ 
such that $u(d+d^{\prime}-\varepsilon)_{+}u^*\in \overline{bAb}$. Note that 
$(d+d^{\prime}-\varepsilon)_{+}=(d-\varepsilon)_{+}+(d^{\prime}-\varepsilon)_{+}$ because of 
$dd^{\prime}=0$. 
Put 
$$
e:=u(d-\varepsilon)_{+}u^* \quad \text{and} \quad 
e^{\prime}:=u(d^{\prime}-\varepsilon)_{+}u^*,
$$ 
then $e$ and $e^{\prime}$ are desired positive elements. 
\end{proof}
\ \\
\textit{Proof of Proposition \ref{pro:non-sep1}.}
(i) We may assume that $\|a\|=1$ and $\|b\|=1$. 
Since $b\notin J$, we have $\tau_{B, \omega}(b)>0$. 
Take a representative $(b_n)_n$ of $b$ such that $\|b_n\|=1$ for any $n\in\mathbb{N}$, 
and choose $\varepsilon_0>0$ such that $\tau_{B, \omega}(b)-\varepsilon_0>0$. 
Since we have 
$$
\lim_{n\to\omega} d_{\tau_{B}}(b_n) \geq \lim_{n\to\omega}\tau_{B}(b_n)=\tau_{B, \omega}(b), 
$$
there exists an element $X_1\in\omega$ such that for any $n\in X_1$, 
$$
d_{\tau_{B}}(b_n) > \tau_{B, \omega}(b)-\varepsilon_0.
$$
By a similar argument as in the proof of \cite[Lemma 3.2]{Sa}, we see that there exists a 
representative $(a_n)_n$ of $a$ such that $a_n\in \overline{b_nBb_n}$ and $\|a_n\|=1$ 
for any $n\in\mathbb{N}$ and $\lim_{n\to\omega}d_{\tau_{B}}(a_n)=0$ 
because of $a\in \overline{(b_n)_nJ(b_n)_n}$. 
Hence there exists an element $X_2\in\omega$ such that for any $n\in X_2$, 
$$
d_{\tau_{B}}(a_n) < \frac{\tau_{B,\omega}(b)-\varepsilon_0}{2}.
$$ 
Note that we have  
$
d_{\tau_{B}}(a_n) < \frac{d_{\tau_{B}}(b_n)}{2} 
$
for any $n\in X_1\cap X_2$.  
Hence Lemma \ref{lem:positive-e} implies that for any $n\in X_1\cap X_2$, there exist positive 
elements $e_n$ and $e_n^{\prime}$ in $\overline{b_nBb_n}$ such that 
$e_ne_n^{\prime}=0$ and $d_{\tau_B}(e_n)=d_{\tau_B}(e_n^{\prime})>d_{\tau_{B}}(a_n)$. 
Since $\overline{b_nBb_n}$ has strict comparison and 
$\overline{b_nBb_n}\subseteq \overline{\mathrm{GL}(\overline{b_nBb_n}^{\sim})}$ 
by \cite{Ror} and \cite{Rob2}, 
Lemma \ref{lem:Pedersen-Rordam} shows that for any $n\in X_1\cap X_2$, 
there exist unitary elements $u_n$ and $v_n$ in $\overline{b_nBb_n}^{\sim}$ such that 
$$
u_n(a_n-1/n)_{+}u_n^*\in \overline{e_nBe_n} \quad \text{and} \quad 
v_n(a_n-1/n)_{+}v_n^*\in \overline{e_n^{\prime}Be_n^{\prime}}. 
$$
Note that $(a_n-1/n)_{+}u_n^*v_n(a_n-1/n)_{+}=0$ for any $n\in X_1\cap X_2$. 
Define $z=(z_n)_n$ and $c=(c_n)_n$ in $B^{\omega}$ by 
$$
z_{n} := \left\{\begin{array}{cl}
0 & \text{if } n\notin X_1\cap X_2   \\
u_{n}^*v_n(a_n-1/n)_{+}^{1/2} & \text{if } n\in X_1\cap X_2
\end{array}
\right. 
$$
and
$$
c_{n} := \left\{\begin{array}{cl}
0 & \text{if } n\notin X_1\cap X_2   \\
u_{n}^*v_n(a_n-1/n)_{+}v_n^*u_n & \text{if } n\in X_1\cap X_2
\end{array}
\right..
$$
It is easy to see that $z, c\in\overline{bB^{\omega}b}$, $z^*z=a$, $zz^*=c$ and $ac=0$. 
Since $\overline{bJb}$ is a closed ideal in $\overline{bB^{\omega}b}$ and $a\in \overline{bJb}$, 
$z$ and $c$ are elements in $\overline{bJb}$. Therefore we obtain the conclusion. 

(ii) 
Note that $B^{\omega}$ has strict comparison (see, for example, \cite[Lemma 1.23]{BBSTWW}). 
Since $a\in J$ and $b\notin J$, we have $d_{\tau_{B, \omega}}(a^{1/5})=0$ and 
$d_{\tau_{B, \omega}}(b)>0$. 
Hence there exists a sequence 
$\{s_N \}_{N\in\mathbb{N}}$ in $B^{\omega}$ such that $\lim_{N\to\infty} \| s_N^*bs_N- a^{1/5}\|=0$. 
Let $d_{N}:=bs_{N}a^{1/5}s_{N}^*b$ and $r_{N}:= s_{N}a^{1/5}$ for any $N\in\mathbb{N}$. 
Then we have $d_{N}\in \overline{bJb}$, $r_{N}\in J$ for any $N\in\mathbb{N}$ 
and 
$$
r_N^*d_Nr_N= a^{1/5}s_{N}^*bs_{N}a^{1/5}s_{N}^*bs_{N}a^{1/5}\to a
$$ as $N\to\infty$. Therefore we obtain the conclusion. 

(iii) Since $B$ is a simple monotracial  $\mathcal{Z}$-stable C$^*$-algebra, 
$B\subseteq \overline{\mathrm{GL}(B^{\sim})}$ by \cite{Ror} and  \cite{Rob2}. Therefore 
we obtain the conclusion by Proposition \ref{pro:non-sep-gl}. 
\qed 
\ \\

If $B$ is unital, then the following lemma is a well-known consequence of Proposition 
\ref{pro:non-sep-gl} and Blackadar's technique (see \cite[II.8.5.4]{Bla}). 

\begin{lem}\label{lem:Bla1}
With notation as above, let $S$ be a separable subset of $B^{\omega}$. 
Then there exists a separable C$^*$-algebra $A$ such that 
$S\subseteq A\subset B^{\omega}$ and $A\subseteq \overline{\mathrm{GL}(A^{\sim})}$. 
\end{lem}
\begin{proof}
We shall show only the case where $B$ is non-unital. 
Let $A_1$ be the C$^*$-subalgebra of $B^{\omega}$ generated by $S$. 
Since $A_1$ is separable, there exists a countable dense subset 
$\{x_k\; |\; k\in\mathbb{N}\}$ of $A_1$. 
By Proposition \ref{pro:non-sep1}.(iii), for any $k, m\in\mathbb{N}$, 
there exist $y_{k, m}\in B^{\omega}$ and $\lambda_{k, m}\in\mathbb{C}\setminus\{0\}$ such that 
$$ 
\| x_{k}- (y_{k, m}+ \lambda_{k, m}1_{(B^\omega)^{\sim}})\|<\frac{1}{m}
$$
and $y_{k, m}+ \lambda_{k, m}1_{(B^\omega)^{\sim}}\in \mathrm{GL}((B^{\omega})^{\sim})$. 
Let $A_2$ be the C$^*$-subalgebra of $B^{\omega}$ generated by $A_1$ and 
$\{y_{k, m}\; |\; k,m\in\mathbb{N}\}$. 
Then we have $A_1\subseteq \overline{\mathrm{GL}(A_2^{\sim})}$. 
Indeed, we have $y_{k, m}+\lambda_{k, m}1_{A_2^{\sim}}\in \mathrm{GL}(A_2^{\sim})$ 
for any $k,m\in\mathbb{N}$ 
because of $\mathrm{Sp}_{A_2}(y_{k,m})\cup\{0\}=\mathrm{Sp}_{B^{\omega}}(y_{k,m})\cup\{0\}$ 
and $\lambda_{k,m}\neq 0$. Since we have $A_1=\overline{\{x_{k}\; |\; k\in\mathbb{N}\}}$ and 
\begin{align*}
\|x_{k}- (y_{k, m}+ \lambda_{k, m}1_{(A_2)^{\sim}})\|
& =\|1_{A_2^{\sim}}x_{k}- 1_{A_2^{\sim}}(y_{k, m}+ \lambda_{k, m}1_{(B^\omega)^{\sim}})\| \\
& \leq \|x_{k}- (y_{k, m}+ \lambda_{k, m}1_{(B^\omega)^{\sim}})\| <\frac{1}{m}
\end{align*}
for any $k,m\in\mathbb{N}$, we have $A_1\subseteq \overline{\mathrm{GL}(A_2^{\sim})}$. 
Repeating this process, we obtain a sequence $\{A_n\}_{n\in\mathbb{N}}$ of separable 
C$^*$-subalgebras of $B^{\omega}$ such that $A_n\subseteq A_{n+1}$ and 
$A_{n}\subseteq \overline{\mathrm{GL}(A_{n+1}^{\sim})}$ for any $n\in\mathbb{N}$. 
Put $A:=\overline{\bigcup_{n=1}^\infty A_n}$. Since we have 
$A_n\subseteq \overline{\mathrm{GL}(A_{n+1}^{\sim})}\subseteq \overline{\mathrm{GL}(A^{\sim})}$ 
for any $n\in\mathbb{N}$ by Proposition \ref{pro:inclusion-gl}, we have
$A\subseteq \overline{\mathrm{GL}(A^{\sim})}$. 
Therefore $A$ is the desired separable C$^*$-algebra. 
\end{proof}

The following lemma is also based on Blackadar's technique. 

\begin{lem}\label{lem:Bla2}
With notation as above, let $\{b_k\; |\; k\in\mathbb{N}\}$ be a countable subset of 
$B^{\omega}\setminus J$ and $S$ a separable subset of $B^{\omega}$. 
Then there exists a separable C$^*$-algebra $A$ such that 
$\{b_k\; |\; k\in\mathbb{N}\}\cup S\subseteq A\subset B^{\omega}$ and 
$\overline{b_k(A\cap J)b_k}$ is full in $A\cap J$ for any $k\in\mathbb{N}$. 
\end{lem}
\begin{proof}
Let $A_1$ be the C$^*$-subalgebra of $B^{\omega}$ generated by $\{b_k\; |\; k\in\mathbb{N}\}$ 
and $S$. Since $A_1$ is separable, there exists  a countable dense subset 
$\{a_l\; |\; l\in\mathbb{N}\}$ of $(A_1\cap J)_{+}$. 
By Proposition \ref{pro:non-sep1}.(ii), for any $k, l, m\in\mathbb{N}$, 
there exist $d_{k ,l, m}\in \overline{b_{k}Jb_{k}}_{+}$ and $r_{k, l, m}\in J$ such that 
$$
\| r_{k, l, m}^*d_{k, l, m}r_{k, l, m}-a_{l}\| < \frac{1}{m}.
$$
Let $A_2$ be the C$^*$-subalgebra of $B^{\omega}$ generated by $A_1$ and 
$\{d_{k ,l, m}, r_{k, l, m}\; |\; k,l,m\in\mathbb{N}\}$. 
Then we have $A_1\cap J\subseteq \overline{(A_2\cap J)b_{k}(A_2\cap J)b_{k}(A_2\cap J)}$ 
for any $k\in\mathbb{N}$ because $A_1\cap J$ is generated by $\{a_l\; |\; l\in\mathbb{N}\}$. 
Repeating this process, we obtain a sequence $\{A_n\}_{n\in\mathbb{N}}$ of separable 
C$^*$-subalgebras of $B^{\omega}$ such that $A_n\subseteq A_{n+1}$ and 
$A_{n} \cap J\subseteq \overline{(A_{n+1}\cap J)b_{k}(A_{n+1}\cap J)b_{k}(A_{n+1}\cap J)}$ 
for any $k, n\in\mathbb{N}$. Put $A:=\overline{\bigcup_{n=1}^\infty A_n}$. 
Since we have $A\cap J= \overline{\bigcup_{n=1}^\infty (A_n\cap J)}$, we see that 
$A$ is the desired separable C$^*$-algebra. 
\end{proof}

By Lemma \ref{lem:Bla1} and Lemma \ref{lem:Bla2}, \cite[II.8.5.3]{Bla} implies the following lemma. 

\begin{lem}\label{lem:Bla}
With notation as above, let $\{b_k\; |\; k\in\mathbb{N}\}$ be a countable subset of 
$B^{\omega}\setminus J$ and $S$ a separable subset of $B^{\omega}$. 
Then there exists a separable C$^*$-algebra $A$ such that 
$\{b_k\; |\; k\in\mathbb{N}\}\cup S\subseteq A\subset B^{\omega}$, 
$A\subseteq \overline{\mathrm{GL}(A^{\sim})}$ and $\overline{b_k(A\cap J)b_k}$ is 
full in $A\cap J$ for any $k\in\mathbb{N}$. 
\end{lem}

We shall construct the separable extension $\eta_0$ of Lemma \ref{lem:qd}. 

Since $\varrho$ is surjective and $\mathcal{D}$ is separable, there exists a separable subset 
$S_0$ of $B^{\omega}$ such that $\overline{\varrho (S_0)}=\Pi(\mathcal{D})$. 
Applying Lemma \ref{lem:Bla1} to $S_0$, we obtain a separable C$^*$-algebra $B_1$ 
such that $S_0\subseteq B_1 \subset B^{\omega}$ and 
$B_1\subseteq\overline{\mathrm{GL}(B_1^{\sim})}$. 
Since $B_1$ is separable, there exist a countable subset 
$\{a_{1,m}\; |\; m\in\mathbb{N}\}$ of $(B_1\cap J)_{+}$ and a countable subset 
$\{b_{1,k}\; |\; k\in\mathbb{N}\}$ of $B_{1+}$ such that 
$$
\overline{\{a_{1,m}\; |\; m\in\mathbb{N}\}}=(B_1\cap J)_{+} \quad \text{and} \quad 
\overline{\{b_{1,k}\; |\; k\in\mathbb{N}\}}=B_{1+}. 
$$
Put $T_{1}:= \{(k, l)\in \mathbb{N}\times\mathbb{N}\; |\; (b_{1,k}-1/l)_{+}\notin J\}$. 
Applying Proposition \ref{pro:non-sep1}.(i) to $(b_{1,k}-1/l)_{+}a_{1,m}(b_{1,k}-1/l)_{+}$ for any 
$(k,l)\in T_1$ and $m\in\mathbb{N}$, there exist 
a positive element $c_{1,1,(k,l),m}$ and an element $z_{1,1,(k,l),m}$ 
in $\overline{(b_{1,k}-1/l)_{+}J(b_{1,k}-1/l)_{+}}$ such that
$$
(b_{1,k}-1/l)_{+}a_{1,m}(b_{1,k}-1/l)_{+}c_{1,1,(k,l),m}=0, 
$$
$$
z_{1,1,(k,l),m}^*z_{1,1,(k,l),m}=(b_{1,k}-1/l)_{+}a_{1,m}(b_{1,k}-1/l)_{+}
$$ 
and 
$$
z_{1,1,(k,l),m}z_{1,1,(k,l),m}^*=c_{1,1,(k,l),m}.
$$
Let $S_2:= B_1\cup \{c_{1,1,(k,l),m}, z_{1,1,(k,l),m}\; |\; (k,l)\in T_1, m\in\mathbb{N}\}$. 
Applying Lemma \ref{lem:Bla} to $\{(b_{1, k}-1/l)_{+}\; |\; (k,l)\in T_1\}$ and $S_2$, we obtain 
a separable C$^*$-algebra $B_2$ such that 
$$
B_1\cup\{c_{1,1,(k,l),m}, z_{1,1,(k,l),m}\; |\; (k,l)\in T_1, m\in\mathbb{N}\} \subseteq B_2
\subset B^{\omega}, 
$$
$B_2\subseteq \overline{\mathrm{GL}(B_2^{\sim})}$ and 
$\overline{(b_{1,k}-1/l)_{+}(B_2\cap J)(b_{1,k}-1/l)_{+}}$ is full in $B_2\cap J$ 
for any $(k,l)\in T_1$. 
By the same way as above, there exist a countable subset 
$\{a_{2,m}\; |\; m\in\mathbb{N}\}$ of $(B_2\cap J)_{+}$ and a countable subset 
$\{b_{2,k}\; |\; k\in\mathbb{N}\}$ of $B_{2+}$ such that 
$$
\overline{\{a_{2,m}\; |\; m\in\mathbb{N}\}}=(B_2\cap J)_{+} \quad \text{and} \quad 
\overline{\{b_{2,k}\; |\; k\in\mathbb{N}\}}=B_{2+},  
$$
and we put $T_{2}:= \{(k, l)\in \mathbb{N}\times\mathbb{N}\; |\; (b_{2,k}-1/l)_{+}\notin J\}$. 
Applying Proposition \ref{pro:non-sep1}.(i) to 
$(b_{i,k}-1/l)_{+}a_{2,m}(b_{i,k}-1/l)_{+}$ for any $1\leq i\leq 2$, $(k,l)\in T_i$ and 
$m\in\mathbb{N}$, there exist a positive element $c_{2,i,(k,l),m}$ and an element $z_{2,i,(k,l),m}$ 
in $\overline{(b_{i,k}-1/l)_{+}J(b_{i,k}-1/l)_{+}}$ such that
$$
(b_{i,k}-1/l)_{+}a_{2,m}(b_{i,k}-1/l)_{+}c_{2,i,(k,l),m}=0, 
$$
$$
z_{2,i,(k,l),m}^*z_{2,i,(k,l),m}=(b_{i,k}-1/l)_{+}a_{2,m}(b_{i,k}-1/l)_{+}
$$ 
and 
$$
z_{2,i,(k,l),m}z_{2,i,(k,l),m}^*=c_{2,i,(k,l),m}.
$$
Let $S_3:=B_2\cup \{c_{2,i,(k,l),m}, z_{2,i,(k,l),m}\; |\; 1\leq i\leq 2, (k,l)\in T_i, m\in\mathbb{N}\}$. 
Applying Lemma \ref{lem:Bla} to $\{(b_{i, k}-1/l)_{+}\; |\; 1\leq i\leq 2, (k,l)\in  T_i\}$ and $S_3$, 
we obtain a separable C$^*$-algebra $B_3$ such that 
$$
B_2\cup\{c_{2,i,(k,l),m}, z_{2,i,(k,l),m}\; |\; 1\leq i \leq 2, (k,l)\in T_i, m\in\mathbb{N}\} \subseteq B_3
\subset B^{\omega}, 
$$
$B_3\subseteq \overline{\mathrm{GL}(B_3^{\sim})}$ and 
$\overline{(b_{i,k}-1/l)_{+}(B_3\cap J)(b_{i,k}-1/l)_{+}}$ is full in $B_3\cap J$ for any 
$1\leq i \leq 2$ and $(k,l)\in T_{i}$. 
Repeating this process, for any $n\in\mathbb{N}$, we obtain  
$$
B_n\subset B^{\omega},  \quad 
\{a_{n,m}\; |\; m\in\mathbb{N}\} \subset (B_{n}\cap J)_{+}, \quad  
\{b_{n,k} \; |\; k\in\mathbb{N}\}\subset B_{n+}, 
$$
$$
T_n\subset \mathbb{N}\times \mathbb{N}, \quad 
\{c_{n,i,(k,l),m}, z_{n,i,(k,l),m}\; |\; 1\leq i\leq n, (k,l)\in T_i, m\in\mathbb{N}\}
$$
such that $B_n$ is separable, 
$$
B_n\subseteq B_{n+1}, \quad B_n\subseteq \overline{\mathrm{GL}(B_n^{\sim})}, \quad
\overline{\{a_{n,m}\; |\; m\in\mathbb{N}\}}=(B_{n}\cap J)_{+},  
$$
$$
\overline{\{b_{n,k} \; |\; k\in\mathbb{N}\}}= B_{n+}, \quad 
T_{n}= \{(k, l)\in \mathbb{N}\times\mathbb{N}\; |\; (b_{n,k}-1/l)_{+}\notin J\}, 
$$
$$
c_{n,i,(k,l),m},\; z_{n,i,(k,l),m}\in \overline{(b_{i,k}-1/l)_{+}(B_{n+1}\cap J)(b_{i,k}-1/l)_{+}},
$$
$$
(b_{i,k}-1/l)_{+}a_{n,m}(b_{i,k}-1/l)_{+}c_{n,i,(k,l),m}=0, 
$$
$$
z_{n,i,(k,l),m}^*z_{n,i,(k,l),m}=(b_{i,k}-1/l)_{+}a_{n,m}(b_{i,k}-1/l)_{+}, 
$$
$$
z_{n,i,(k,l),m}z_{n,i,(k,l),m}^*=c_{n,i,(k,l),m}
$$
and $\overline{(b_{i,k}-1/l)_{+}(B_{n+1}\cap J)(b_{i,k}-1/l)_{+}}$ is full in $B_{n+1}\cap J$ 
for any $1\leq i\leq n$ and $(k,l)\in T_{i}$. Define 
$$
B_0:= \overline{\bigcup_{n=1}^\infty B_n}, \quad J_0:= B_{0}\cap J \quad \text{and} \quad 
M_0:=\varrho (B_0).
$$
Then 
$$
\xymatrix{
\eta_0: & 0 \ar[r]  & J_0 \ar[r] & B_0 \ar[r]^{\varrho} & M_0 \ar[r] & 0
}
$$
is a separable extension and $\Pi(\mathcal{D})\subseteq M_0$. 
Corollary \ref{cor:inclusion-gl} implies $B_0\subseteq \overline{\mathrm{GL}(B_0^{\sim})}$ 
since we have $B_n\subseteq \overline{\mathrm{GL}(B_n^{\sim})}$ for any 
$n\in\mathbb{N}$. Furthermore, for any $i\in\mathbb{N}$ and $(k,l)\in T_{i}$, 
$\overline{(b_{i,k}-1/l)_{+}J_0(b_{i,k}-1/l)_{+}}$ is full in $J_0$ by a similar argument 
as in the proof of Lemma \ref{lem:Bla2}. 
Note that for any $n_0\in\mathbb{N}$, 
$$
J_{0+}=\overline{\bigcup_{n=n_0}^\infty \{a_{n,m}\; |\; m\in\mathbb{N}\}}\quad \text{and} \quad 
B_{0+}=\overline{\bigcup_{n=n_0}^\infty \{b_{n,k}\; |\; k\in\mathbb{N}\}}.
$$
We shall show that $J_0$ is stable and $\eta_0$ is purely large. 

\ \\
\textit{Proof of stability of $J_0$.}
Let $a\in J_{0+}\setminus \{0\}$ and $\varepsilon>0$. 
Set 
$$
\varepsilon^{\prime}:=\min\left\{\frac{\varepsilon}{2\| a\|},\; \sqrt{\frac{\varepsilon}{2}},\; 
\varepsilon \right\}.
$$ 
Since $B_0$ is separable, there exists an approximate unit 
$\{h_n\}_{n\in\mathbb{N}}$ for $B_0$. 
Note that $h_n\notin J$ for sufficiently large $n$ because of $M_0\neq \{0\}$. 
Hence there exists $N\in\mathbb{N}$ such that $h_{N}\notin J$ and 
$\| h_Nah_N- a\|< \varepsilon^{\prime}/2$. 
Since $B_{0+}=\overline{\bigcup_{n=1}^\infty \{b_{n,k}\; |\; k\in\mathbb{N}\}}$, 
for any $l\in\mathbb{N}$, there exist $n(l)$ and $k(l)$ in $\mathbb{N}$ such that 
$$
\| h_{N}- b_{n(l),k(l)} \|< \frac{1}{l}.
$$
Note that $(b_{n(l),k(l)}-1/l)_{+}\to h_{N}$ as $l\to \infty$ because we have 
$$
\|h_{N} -(b_{n(l),k(l)}-1/l)_{+} \| 
\leq \| h_{N}- b_{n(l),k(l)} \| + \|b_{n(l),k(l)} -(b_{n(l),k(l)}-1/l)_{+}\| \\ 
<\frac{2}{l}.
$$
Hence there exists $l_0\in\mathbb{N}$ such that 
$(b_{n(l_0),k(l_0)}-1/l_0)_{+}\notin J$, that is, $(k(l_0), l_0)\in T_{n(l_0)}$ and 
$$
\| a - (b_{n(l_0),k(l_0)}-1/l_0)_{+}a(b_{n(l_0),k(l_0)}-1/l_0)_{+} \|< \frac{\varepsilon^{\prime}}{2}.
$$
Since $J_{0+}=\overline{\bigcup_{n=n(l_0)}^\infty \{a_{n,m}\; |\; m\in\mathbb{N}\}}$, 
there exist $n_0\geq n(l_0)$ and $m_0\in\mathbb{N}$ such that 
$$
\| a- a_{n_{0}, m_0}\| < \frac{\varepsilon^{\prime}}{2\|b_{n(l_0),k(l_0)}\|^2}.
$$
Put $a^{\prime}:=  (b_{n(l_0),k(l_0)}-1/l_0)_{+}a_{n_0, m_0}(b_{n(l_0),k(l_0)}-1/l_0)_{+}$. Then
$$
\| a- a^{\prime}\| < \varepsilon^{\prime}\leq \varepsilon.
$$
By construction of $B_0$ and $J_0$, there exist 
$$
z=z_{n_0, n(l_0), (k(l_0), l_0), m_0},\;  c=c_{n_0, n(l_0), (k(l_0), l_0), m_0}\in J_0
$$ 
such that 
$a^{\prime}c=0$, $z^*z=a^{\prime}$ and $zz^*=c$. Hence $a^{\prime}\sim c$ and 
$$
\| ac\| =\|ac-a^{\prime}c\| \leq \|a-a^{\prime}\| \|c\|= \|a-a^{\prime}\|\|a^{\prime}\|
<\varepsilon^{\prime}(\| a\| +\varepsilon^{\prime})\leq \varepsilon.
$$
Therefore $J_0$ is stable by Hjelmborg and R\o rdam's characterization 
(Theorem \ref{thm:HR}). 
\qed 

\ \\
\textit{Proof of purely largeness of $\eta_0$.} 
Let $x\in B_0\setminus J_0$. Note that we have $xx^*\notin J$. 
Since $B_{0+}=\overline{\bigcup_{n=1}^\infty \{b_{n,k}\; |\; k\in\mathbb{N}\}}$, 
for any $l\in\mathbb{N}$, there exist $n(l)$ and $k(l)$ in $\mathbb{N}$ such that 
$$
\| xx^*- b_{n(l),k(l)} \|< \frac{1}{2l}.
$$
By a similar argument as in the proof of stability of $J_0$,  
there exists $l_0\in\mathbb{N}$ such that $(b_{n(l_0),k(l_0)}-1/l_0)_{+}\notin J$, that is, 
$(k(l_0), l_0)\in T_{n(l_0)}$. On the other hand, \cite[Lemma 2.2]{KR2} implies that 
$(b_{n(l_0),k(l_0)}-1/2l_0)_{+}$ is Cuntz smaller than $xx^*$. 
Since we have $B_0\subseteq \overline{\mathrm{GL}(B_0^{\sim})}$, there exists a unitary element 
$u$ in $B_0^{\sim}$ such that 
$$
u(b_{n(l_0),k(l_0)}-1/l_0)_{+}u^*=
u((b_{n(l_0),k(l_0)}-1/2l_0)_{+}-1/2l_0)_{+}u^*\in \overline{xx^*B_0xx^*}=\overline{xB_0x^*}
$$ 
by Lemma \ref{lem:Pedersen-Rordam}. Put
$$
C:=u\overline{(b_{n(l_0),k(l_0)}-1/l_0)_{+}J_0(b_{n(l_0),k(l_0)}-1/l_0)_{+}}u^*\subseteq 
\overline{xJ_0x^*},
$$
then $C$ is full in $J_0$ because $\overline{(b_{n(l_0),k(l_0)}-1/l_0)_{+}J_0(b_{n(l_0),k(l_0)}-1/l_0)_{+}}$
is full in $J_0$. We shall show that $C$ is stable. 
Let $a\in C_{+}\setminus \{0\}$ and $\varepsilon>0$. 
Set 
$$
\varepsilon^{\prime}:=\min\left\{\frac{\varepsilon}{2\| a\|},\; \sqrt{\frac{\varepsilon}{2}},\; 
\varepsilon \right\}.
$$ 
By the definition of $C$ and 
$J_{0+}=\overline{\bigcup_{n=n(l_0)}^\infty \{a_{n,m}\; |\; m\in\mathbb{N}\}}$, 
there exist $n_0\geq n(l_0)$ and $m_0\in\mathbb{N}$ such that 
$$
\| a - u(b_{n(l_0),k(l_0)}-1/l_0)_{+}a_{n_0, m_0}(b_{n(l_0),k(l_0)}-1/l_0)_{+}u^* \|< \varepsilon^{\prime}\leq 
\varepsilon.
$$
Put $a^{\prime}=u(b_{n(l_0),k(l_0)}-1/l_0)_{+}a_{n_0, m_0}(b_{n(l_0),k(l_0)}-1/l_0)_{+}u^*
\in C$, then 
$\|a -a^{\prime}\|< \varepsilon^{\prime}\leq \varepsilon$. 
By construction of $B_0$ and $J_0$, there exist elements 
$$
z_{n_0, n(l_0), (k(l_0), l_0), m_0},\;  c_{n_0, n(l_0), (k(l_0), l_0), m_0}
$$
in 
$
\overline{(b_{n(l_0),k(l_0)}-1/l_0)_{+}J_0(b_{n(l_0),k(l_0)}-1/l_0)_{+}}
$ 
such that 
$$
u^*a^{\prime}uc_{n_0, n(l_0), (k(l_0), l_0), m_0}=0, \quad 
z_{n_0, n(l_0), (k(l_0), l_0), m_0}^*z_{n_0, n(l_0), (k(l_0), l_0), m_0}
=u^*a^{\prime}u
$$
and 
$$
z_{n_0, n(l_0), (k(l_0), l_0), m_0} z_{n_0, n(l_0), (k(l_0), l_0), m_0}^*=c_{n_0, n(l_0), (k(l_0), l_0), m_0}.
$$
Put $c:=uc_{n_0, n(l_0), (k(l_0), l_0), m_0}u^*$. It is easy to see that 
$c\in C$, $a^{\prime}c=0$ and 
$$
c\sim c_{n_0, n(l_0), (k(l_0), l_0), m_0} \sim u^*a^{\prime}u\sim a^{\prime} \quad \text{in}\quad B_0.
$$
Since $C$ is a hereditary C$^*$-subalgebra of $B_0$ and $a^{\prime},c\in C$, 
we see that $a^{\prime}$ is Murray-von Neumann equivalent to $c$ in $C$. 
Therefore, the same argument as in the proof of stability of $J_0$ shows 
$\| ac\|< \varepsilon$, and $C$ is stable. Consequently, $\eta_0$ is a purely large extension. 
\qed 
\ \\

Therefore we obtain the following lemma. 

\begin{lem}
With notation as above, there exist separable C$^*$-subalgebras 
$J_{0}\subset J$, $B_{0}\subset B^{\omega}$ and $M_{0}\subset M$ 
such that $J_0$ is stable, 
$$
\xymatrix{
\eta_0: & 0 \ar[r]  & J_0 \ar[r] & B_0 \ar[r]^{\varrho|_{B_0}} & M_0 \ar[r] & 0
}
$$
is a purely large extension and $\Pi (\mathcal{D})\subset M_0$.
\end{lem}

Consequently, we obtain the following theorem by Lemma \ref{lem:kp}, Lemma \ref{lem:qd} 
and the lemma above.

\begin{thm}\label{thm:existence}
Let $\mathcal{D}$ be a simple separable nuclear monotracial $M_{2^{\infty}}$-stable 
C$^*$-algebra which is $KK$-equivalent to $\{0\}$ and $B$ a simple separable exact 
monotracial $\mathcal{Z}$-stable C$^*$-algebra. 
Then there exists a trace preserving homomorphism from $\mathcal{D}$ to $B$. 
\end{thm}

\begin{rem}
Actually, we need not assume that $\mathcal{D}$ is $M_{2^{\infty}}$-stable in the theorem above. 
Indeed, define a homomorphism $\varphi$ from $\mathcal{D}$ to 
$\mathcal{D}\otimes M_{2^{\infty}}$ by $\varphi(a)=a\otimes 1$. Then $\varphi$ is a trace 
preserving homomorphism from $\mathcal{D}$ to $\mathcal{D}\otimes M_{2^{\infty}}$. 
By the theorem above, there exists a trace preserving homomorphism $\psi$ 
from $\mathcal{D}\otimes M_{2^{\infty}}$ to $B$. Then $\psi\circ \varphi$ is 
is a trace preserving homomorphism from $\mathcal{D}$ to $B$. 
\end{rem}

The following corollary is an immediate consequence of the theorem above. 

\begin{cor}\label{cor:existence-w}
Let $B$ a simple separable exact monotracial $\mathcal{Z}$-stable C$^*$-algebra. 
Then there exists a trace preserving homomorphism from $\mathcal{W}$ to $B$. 
\end{cor}

The injective II$_1$ factor can embed unitally into every II$_1$ factor. 
Hence the following question is natural and interesting. 

\begin{que}
(1) Let $B$ be a simple monotracial infinite-dimensional C$^*$-algebra. 
Does there exist a trace preserving homomorphism from $\mathcal{W}$ to $B$?
\ \\
(2) Let $B$ be a simple non-type I C$^*$-algebra. Does there exists a (non-zero) homomorphism 
from $\mathcal{W}$ to $B$?
\end{que}

Note that Dadarlat, Hirshberg, Toms and Winter \cite{DHTW} showed that 
there exists a unital simple separable nuclear infinite-dimensional C$^*$-algebra $B$ such that 
$\mathcal{Z}$ does not embed unitally into $B$.

\section{Characterization of $\mathcal{W}$}\label{sec:main}

In this section we shall show that if $\mathcal{D}$ is 
a simple separable nuclear monotracial $M_{2^{\infty}}$-stable C$^*$-algebra 
which is $KK$-equivalent to $\{0\}$, then $\mathcal{D}$ is isomorphic to $\mathcal{W}$. 
Also, we shall characterize $\mathcal{W}$ by using properties of $F(\mathcal{W})$. 

\begin{thm}\label{thm:main}
Let $\mathcal{D}$ be a simple separable nuclear monotracial $M_{2^{\infty}}$-stable C$^*$-algebra 
which is $KK$-equivalent to $\{0\}$. Then $\mathcal{D}$ is isomorphic to $\mathcal{W}$. 
\end{thm}
\begin{proof}
By Theorem \ref{thm:existence} and Corollary \ref{cor:existence-w}, there exist 
trace preserving homomorphisms $\varphi$ and $\psi$ from $\mathcal{D}$ to $\mathcal{W}$ 
and from $\mathcal{W}$ and $\mathcal{D}$, respectively. 
Since $\mathcal{D}$ and $\mathcal{W}$ have property W by Corollary \ref{cor:property-w}, 
Theorem \ref{thm:uniqueness-i} implies that $\psi\circ \varphi$ and $\varphi\circ \psi$ are 
approximately inner. 
Therefore $\mathcal{D}$ is isomorphic to $\mathcal{W}$ 
by Elliott's approximate intertwining argument \cite{Ell0} (see also \cite[Corollary 2.3.4]{Ror1}). 
\end{proof}

The following corollary is an immediate consequence of the theorem above. 

\begin{cor}\label{cor:w}
(i) If $A$ is a simple separable nuclear monotracial C$^*$-algebra, then  
$A\otimes\mathcal{W}$ is isomorphic to $\mathcal{W}$.  In particular, 
$\mathcal{W}\otimes\mathcal{W}$ is isomorphic to $\mathcal{W}$. \ \\
(ii) For any non-zero positive element $h$ in $\mathcal{W}$, $\overline{h\mathcal{W}h}$ 
is isomorphic to $\mathcal{W}$. 
\end{cor}

Following the definition in \cite{LN}, 
we say that a C$^*$-algebra $A$ is \textit{$\mathcal{W}$-embeddable} 
if there exists an injective homomorphism from $A$ to $\mathcal{W}$. 

\begin{lem}\label{lem:w-embed}
Let $A$ be a monotracial $\mathcal{W}$-embeddable C$^*$-algebra. Then 
there exists a trace preserving homomorphism from $A$ to $\mathcal{W}$. 
\end{lem}
\begin{proof}
By the assumption, there exists an injective homomorphism $\varphi$ from $A$ to $\mathcal{W}$. 
Let $s$ be a strictly positive element in $A$. (Note that $A$ is separable because  
$A$ is $\mathcal{W}$-embeddable.) Since $\varphi$ is injective, $\varphi (s)$ is a non-zero 
positive element. Corollary \ref{cor:w} implies that there exists a isomorphism $\Phi$ from 
$\overline{\varphi(s)\mathcal{W}\varphi(s)}$ onto $\mathcal{W}$. 
Note that $\varphi$ can be regarded as a homomorphism from $A$ to 
$\overline{\varphi(s)\mathcal{W}\varphi(s)}$. 
Define $\psi := \Phi\circ \varphi$,  then $\psi$ is a trace preserving homomorphism from 
$A$ to $\mathcal{W}$. 
\end{proof}

The following theorem is a characterization of $\mathcal{W}$. 

\begin{thm}\label{thm:main2}
Let $D$ be a simple separable nuclear monotracial C$^*$-algebra. Then 
$D$ is isomorphic to $\mathcal{W}$ if and only if $D$ has property W and is 
$\mathcal{W}$-embeddable, that is, 
$D$ satisfies the following properties: 
\ \\
(i) for any $\theta\in [0,1]$, there exists a projection $p$ in $F(D)$ such that 
$\tau_{D, \omega}(p)=\theta$, \ \\
(ii) if $p$ and $q$ are projections in $F(D)$ such that $0<\tau_{D, \omega}(p)=\tau_{D, \omega}(q)$, 
then  $p$ is Murray-von Neumann equivalent to $q$,  \ \\
(iii) there exists an injective homomorphism from $D$ to $\mathcal{W}$. 
\end{thm}
\begin{proof}
The only if part is obvious by Corollary \ref{cor:property-w}. We shall show the if part. 
Since $D$ is $\mathcal{W}$-embeddable, there exists a trace preserving homomorphism $\varphi$
from $D$ to $\mathcal{W}$ by Lemma \ref{lem:w-embed}. 
Lemma \ref{lem:w-car-stable} implies that $D$ is $\mathcal{Z}$-stable because $D$ has 
property W. 
Hence there exists a trace preserving homomorphism $\psi$ from $\mathcal{W}$ to 
$D$ by Corollary \ref{cor:existence-w}. The rest of proof is same as the proof of 
Theorem \ref{thm:main}. 
\end{proof}

We think that every simple separable nuclear monotracial  C$^*$-algebra with property W 
ought to be $\mathcal{W}$-embeddable. 
Note that every simple separable nuclear monotracial C$^*$-algebra with property W is 
stably projectionless by \cite[Remark 2.13]{Kir2} and a similar argument as in the proof 
of \cite[Corollary 5.9]{Na3}. 
Hence an affirmative answer to the following question, which can be regarded as an analogue of 
Kirchberg's embedding theorem \cite{KP}, would imply this. 

\begin{que}
Let $A$ be a simple separable exact stably projectionless monotracial C$^*$-algebra. Assume that 
$\tau_{A}$ is amenable. Is $A$ $\mathcal{W}$-embeddable?
\end{que}

Note that we need to assume that $\tau_A$ is amenable because 
$\pi_{\tau_{\mathcal{W}}}(\mathcal{W})^{''}$ is the injective II$_1$ factor.

\end{document}